\documentclass[journal]{article}

\usepackage{graphicx,color,amsfonts,amsmath,mathrsfs,hyperref}          

\usepackage{lipsum}
\usepackage{amsfonts}
\usepackage{graphicx}
\usepackage{epstopdf}
\usepackage{algorithmic}
\usepackage{cancel}
\usepackage{fullpage}

\usepackage{physics}
\usepackage{amsmath}
\usepackage{tikz}
\usepackage{mathdots}
\usepackage{yhmath}
\usepackage{cancel}
\usepackage{color}
\usepackage{siunitx}
\usepackage{array}
\usepackage{multirow}
\usepackage{amssymb}
\usepackage{gensymb}
\usepackage{tabularx}
\usepackage{extarrows}
\usepackage{booktabs}
\usetikzlibrary{fadings}
\usetikzlibrary{patterns}
\usetikzlibrary{shadows.blur}
\usetikzlibrary{shapes}

\newcommand{\R}{\mathbb{R}}
\newcommand{\N}{\mathbb{N}}

\newtheorem{theorem}{Theorem}
\newtheorem{lemma}{Lemma}
\newtheorem{corollary}[theorem]{Corollary}
\newtheorem{proposition}[lemma]{Proposition}
\newtheorem{definition}[lemma]{Definition}

\newtheorem{remark}[lemma]{Remark}
\newtheorem{example}{Example}

\newenvironment{proof}{{\it Proof.}~~}{\hfill$\square$}

\newcommand{\D}{\mathcal{D}}
\newcommand{\diam}[1]{\mathcal{D}(#1)}
\newcommand{\ctr}[2]{M_{#1#2}}

\begin{document}


\title{Exponential convergence of multiagent systems with lack of connection} 


\author{Fabio Ancona, Mohamed Bentaibi, Francesco Rossi\thanks{F. Ancona is with Dipartimento di Matematica ``Tullio Levi-Civita'', Universit\`a\ degli Studi di Padova, Via Trieste 63, 35121 Padova, Italy. M. Bentaibi is with d-fine s.r.l., Piazza Sigmund Freud, 1, 20124 Milano, Italia. F. Rossi is with Dipartimento di Culture del Progetto, Universit\`a\ Iuav di Venezia, 30135 Venezia, Italy. F. Ancona and F. Rossi are members of Gruppo Nazionale per l'Analisi Matematica, la Probabilit\`a e le loro Applicazioni (GNAMPA) of the Istituto Nazionale di Alta Matematica (INdAM).}}
\maketitle


\begin{abstract}                          
Finding conditions ensuring consensus, i.e. convergence to a common value, for a networked system is of crucial interest, both for theoretical reasons and applications. This goal is harder to achieve when connections between agents are temporarily lost.

Here, we prove that known conditions (introduced by Moreau) ensure an exponential convergence to consensus, with explicit rate of convergence. The key result is related to the length of the graph (i.e. the number of connections to reach a common agent): if this is large, then convergence is slow.

This general result also provides conditions for convergence of second-order cooperative systems with lack of connections.
 \end{abstract}


 \maketitle

\section{Introduction and main results}

In recent years, the study of multi-agent systems  has drawn a huge interest in the control community. The reasons behind the impressive rise of this research topic lie in the simplicity of definition, robustness of results, extent of  applications. For general books on this topic, see e.g. \cite{bullo2009distributed,nedich2015convergence,mesbahi2010graph}. 

Among multi-agent systems, the setting of cooperative systems plays a central role. Very generally, first-order cooperative systems (with binary interactions) are of the following form:
\begin{eqnarray}\label{e-ODE-noM}
    \dot{x}_i(t) = \frac{\lambda_i(x)}{N}\sum_{j=1}^{N} \phi_{ij}(x_i,x_j) (x_j(t) - x_i(t)).
\end{eqnarray}
This dynamics describes the evolution of $N \geq 2$ agents on an Euclidean space $\R^d$, where the position $x_i(t) \in \R^d$ may represent opinion on different topics, velocity or other attributes of agent $i$ at time $t$. The solution is then $x(t)=(x_1(t),\ldots,x_N(t))\in \R^{Nd}$. The term $\lambda_i$ is a positive scaling parameter, which role will be discussed in detail in Section \ref{s-models} below. The (nonlinear) influence functions 
\begin{equation}
\label{e-phi}
\phi_{ij}:=\phi_{ij}(x_i,x_j) :\R^d\times \R^{d}\to\R
\end{equation} are used to quantify the influence of agent $j$ on agent $i$, where $i,j\in\{1,\ldots,N\}$. We assume them to be positive, that corresponds to cooperativity of the system, i.e. to the fact that influence between agents tends to bring them closer. Models of this kind are now ubiquitous in networked systems and their applications. One of the most influential is certainly the bounded confidence model for opinion formation, first described (in discrete time) by Hegselmann and Krause in \cite{HK}, see also \cite{rateHK,rateHK2}. The natural goal for such kind of systems is to reach consensus: there exists a common $x^*$ such that $\lim_{t\to +\infty} x_i(t)=x^*$ for all agents $i \in \{1,\ldots,N\}$.

This idea of cooperation has also been extended to second-order systems, in which each agent is described by a pair $(x_i,v_i)$ of position-velocity variables. The dynamics is then given by 
\begin{eqnarray}\label{e-ODE-2nd-noM}
\begin{cases}
    \dot{x}_i(t) = v_i\\
\dot v_i=    \frac{\lambda_i(x)}{N}\sum_{j=1}^{N} \phi_{ij}(x_i,x_j) (v_j(t) - v_i(t))
\end{cases} 
\end{eqnarray}
for $ i,j \in\{1,\ldots,N\}$. As in \eqref{e-phi}, the interaction function $\phi_{ij}$ is non-negative and depends on the positions $x_i,x_j$, but here it is a multiplicative term for the velocity variable. The most famous example of this kind of dynamics was presented by Cucker and Smale in \cite{CS}. A first natural goal for this dynamics is alignment: the velocity variables converge to a common value $v^*$ (i.e. $\lim_{t\to +\infty} v_i(t)=v^*$ for all agents $i \in \{1,\ldots,N\})$. A second goal is the so-called flocking: one has both alignment and uniform boundedness of distances\footnote{i.e. there exists $K>0$ such that $|x_i(t)-x_j(t)|< K$ for all $t\geq 0$ and $i,j \in \{1,\ldots,N\}$}.\\

From the modelling point of view, each agent is expected to communicate with its neighbours through a \textit{network topology}, influenced by sensor characteristics and the environment. While the easiest scenario involves a fixed network topology (e.g. \cite{watts1998collective,olfati2007consensus}), practical situations often involve dynamic changes, due to factors like communication dropouts, security concerns, or intermittent actuation. In this setting, potential lack of connection between agents occurs; therefore, it becomes crucial to investigate whether consensus can still be achieved or not.  For first-order systems, we model this scenario as follows:
\begin{eqnarray}\label{e-ODE}
    \dot{x}_i(t) = \frac{\lambda_i(x)}{N}\sum_{j=1}^{N} M_{ij}(t)\phi_{ij}(x_i,x_j)(x_j(t) - x_i(t))
\end{eqnarray}
for $i \in \{1,\ldots,N\}$. The terms $M_{ij}:[0,+\infty)\mapsto [0,1]$ represent the weight given to the (directed) connection of agent $j$ with agent $i$. They encode the time-varying network topology and account for lack of connection (e.g., when they vanish). 
Given the connections $M_{ij}$, it is natural to build an associated directed graph $G(t)$ where nodes are the $N$ agents and the arrow $i\to j$ exists only when $M_{ij}$ is ``sufficiently strong''. While several different concepts have been proposed in the literature (see Section \ref{s-comparison} below), in this article we focus on the following integral condition.
 \begin{definition}[$(T,\mu)$-connectivity graph] \label{d-conngraph} Let $T,\mu>0$. For each $t\geq 0$, define the $(T,\mu)$-connectivity graph associated to the connection functions $M_{ij}$ as the following time-dependent oriented graph $G(t)$:
\begin{itemize}
\item the set of nodes is the set of indexes $\{1,\ldots,N\}$
\item the arrow $i\to j$ exists if it holds $$\frac1T \int_t^{t+T}M_{ij}(s)\,ds\geq \mu.$$
\end{itemize}
\end{definition}

The key result by Moreau \cite{moreau2004stability} can be stated as follows: if $G(t)$ is constant with respect to time and  admits a globally reachable node, then the system \eqref{e-ODE} exponentially converges to consensus for any initial condition, i.e. there exist $K,\delta>0$ such that the solution satisfies $$|x(t)|\leq K \exp(-\delta t) |x(0)|.$$ Yet, an explicit estimate of the parameters $K,\delta$ is not provided.

The main goal of this article is to estimate these parameters under the Moreau condition. The interest of this result is that it provides formal evidence to the following intuition: if all paths to the globally reachable node are ``short'', then exponential convergence is fast; otherwise, it may be slow for suitable choices of initial conditions and connection functions.

With this idea, we recall the concept of distance between nodes and length of the graph.
\begin{definition} Let $G$ be a graph with a globally reachable node $I$. Given an index $i\in \{1,\ldots,N\}$, we denote by $d(i,I)$ the (directed) distance from $i$ to $I$, that is the minimal length of a path from $i$ to $I$, i.e.
\begin{equation}
\label{e-distance}
d(i,I):= \min \Big\{ n\mbox{~s.t.~} i\to j_1\to \ldots j_n=I\mbox{~path of~}G\Big\}.
\end{equation}
We denote by $d(G,I)$ the length of the graph $G$ to $I$, i.e. 
\begin{equation}\label{e-length}
d(G,I):=\max_{i\in\{1,\ldots,N\}} d(i,I).
\end{equation}
\end{definition}
We can now provide the main statement, first for the linear case and then for the general one.

\begin{theorem}[Main result - 1st order linear systems] \label{t-main} Let the system \eqref{e-ODE} be given with $\lambda_i=\phi_{ij}\equiv 1$. Assume the following hypothesis:
\begin{description}
\item[(H1) \label{hyp:M}] All weights $M_{ij}:[0,+\infty)\to [0,1]$ are 
$\mathscr{L}^1$-measurable.
\end{description}

Fix $T,\mu>0$ and consider the $(T,\mu)$-connectivity graph $G(t)$ associated to the $M_{ij}(t)$. Assume that there exists a constant graph $G^*$ such that all its arrows are contained in $G(kT)$ for all $k\in\N$. Assume moreover that $G^*$ admits a globally reachable node $I$ and has length $d^*:=d(G^*,I)$. Then, for all $n\in\mathbb{N}$, for $\tau=d^*T$, it  holds
\begin{eqnarray}
\label{diam-est-ntau}
\max_{i,j}|x_i(n\tau)-x_j(n\tau)|\leq C^n \max_{i,j}|x_i(0)-x_j(0)|,\label{e-stimadiamThm}
\end{eqnarray}
with 
\begin{equation}
C=1-\frac12\left( \tfrac{\mu T}{N+\mu T}    \right)^{d^*}\exp(-2d^*T). \label{e-C}
\end{equation}
Moreover, for $\tau=d^*T$, for all $t\in[0,\tau]$ and $n\in\mathbb{N}$ it holds
\begin{eqnarray}
&&\max_{i,j}|x_i(n\tau+t)-x_j(n\tau+t)|\leq \phi(t) C^n\max_{i,j}|x_i(0)-x_j(0)|\label{e-stimaexp}
\end{eqnarray}
with
\begin{eqnarray*}
\hspace{-6mm}&&\phi(t):= \begin{cases}
1&~\mbox{for $t\in[0,\tau-\delta]$}\\
C\exp\left(2\frac{N-1}N(\tau-t)\right)&~\mbox{for $t\in[\tau-\delta,\tau]$}
\end{cases}
\end{eqnarray*}
where $\delta:=-\frac{N}{2(N-1)}\log(C)$.
\end{theorem}
\begin{corollary} \label{c-main} Let the system \eqref{e-ODE} be given. Assume that both (H1) in Theorem \ref{t-main} and the following hold:
\begin{description}
\item[(H2) \label{hyp:Lip}] Both the weight functions $$\lambda_i(x):\R^{Nd}\to \R^+$$ and the binary interaction functions $$\phi_{ij}(x_i,x_j):\R^d\times\R^d\to \R^+$$ are Lipschitz continuous, strictly positive.
\end{description}

Define
\begin{eqnarray}\underline{m}&:=&\min\left\{\lambda_i(x) \phi_{ij}(x_i,x_j)\mbox{~~s.t.~} |x_k| \leq  \max_i|x_i(0)|\mbox{ for all }k\in\{1,\ldots,N\}\right\},\label{e-um}\\
\overline{m}&:=&\max\left\{\lambda_i(x) \phi_{ij}(x_i,x_j)\mbox{~~s.t.~} |x_k| \leq  \max_i|x_i(0)|\mbox{ for all }k\in\{1,\ldots,N\}\right\},\label{e-om}
 \end{eqnarray}
 
 Define the $(T,\mu)$-connectivity graph and assume that there exits a constant graph $G^*$ with length
 $d^*$ as in Theorem \ref{t-main}. Then, estimates \eqref{e-stimadiamThm}-\eqref{e-stimaexp} hold, with $\tau=d^*T/\overline{m}$, 
$$C=1-\frac12\left( \tfrac{\underline{m}\mu T}{N+\underline{m}\mu T}    \right)^{d^*}\exp(-2d^*T\overline{m}),\qquad
\phi(t):= \begin{cases}
1&~\mbox{for $t\in[0,\tau-\delta]$}\\
C\exp\left(2\frac{N-1}N(\tau-t)\right)&~\mbox{for $t\in[\tau-\delta,\tau]$}
\end{cases}$$
where $\delta:=-\frac{N}{2(N-1) \,\overline{m}}\log(C)$.
\end{corollary}

One of the by-products of our main result is that it directly provides a generalization of the Moreau condition to the case of time-varying graphs $G(t)$. Moreover, estimates are provided for other conditions that have been proposed in the literature, such as the Persistent Excitation \cite{narendra2012stable,ChSi2010,ChSi2014,ren2008distributed,tang2020bearing,manfredi2016criterion,bonnet2021consensus,anderson2016convergence,chaillet2008uniform}
 and the Integral Scrambling Coefficient conditions \cite{bonnet2023consensus,bonnet2022consensus}. See more details in Section \ref{s-consequences} below.

 A very interesting generalization of our results about the system \eqref{e-ODE} is given by second-order system, that we write in the following form from now on:
 \begin{eqnarray}\label{e-ODE-2nd}
\begin{cases}
    \dot{x}_i(t) = v_i\\
\dot v_i=    \frac{1}{N}\sum_{j=1}^{N} M_{ij}(t)\phi(|x_i-x_j|) (v_j(t) - v_i(t)).
\end{cases}
\end{eqnarray}
Here, we have fixed the coefficients $\lambda_i$  constantly equal to 1 and  the interaction functions $\phi_{ij}$ do not depend on the agents and their value is function of the distance only. Several generalizations are possible with the same techniques described here, only with more complicated statements.

For these systems, we present here our second main result of this article 
\begin{theorem}[Main result - 2nd order systems]\label{t-main2} Let the system \eqref{e-ODE-2nd} be given. Assume that both (H1) in Theorem \ref{t-main} and the following hold:
\begin{description}
\item[(H2') \label{hyp:Lip2}] The binary interaction function $$\phi(r):[0,+\infty)\to \R^+$$ is Lipschitz continuous, strictly positive, nonincreasing.
\end{description}

Define the $(T,\mu)$-connectivity graph and assume that there exits a constant graph $G^*$ with length $d^*$ as in Theorem \ref{t-main}.  Define the diameters in the position and velocity variables, respectively:
\begin{eqnarray}
&&\D_X(t):=\max_{ij}|x_i(t)-x_j(t)|,\\
&&\D_V(t):=\max_{ij}|v_i(t)-v_j(t)|.
\end{eqnarray}
Then, the following estimate holds
\begin{eqnarray}
\D_V((n+1)\tau)&\leq&C(n\tau)\D_V(n\tau),\label{e-stimaV}
\end{eqnarray}
 where $\tau= d^*T/\phi(0)$ and 
\begin{eqnarray*}
&&C(n\tau):= 1-\frac12\left( \tfrac{\phi(\D_X(0)+(n+1)\tau \D_V(0))\mu T}{N+\phi(\D_X(0)+(n+1)\tau \D_V(0))\mu T}    \right)^{d^*}
\exp(-2\tfrac{N-1}{N}d^*T\phi(0)).
\end{eqnarray*}

It also holds:
\begin{itemize}
\item if $\phi(r)^{d^*}$ is nonintegrable at $+\infty$, i.e. $\int_R^{+\infty}\phi(r)^{d^*}\,dr=+\infty$ for some $R>0$, then $\lim_{t\to+\infty} \D_V(t)=0$, i.e. alignment occurs;
\item if it moreover holds $$
\lim_{n\to+\infty} n  \phi(n\tau \D_V(0))^{d^*}>\exp(2\tfrac{N-1}{N}d^*T\phi(0))\left(\tfrac{N}{\mu T}\right)^{d^*},
$$
e.g. when $\lim_{r\to+\infty}r\phi(r)^{d^*}=+\infty$, then flocking occurs, i.e. alignment occurs and $\D_X(t)$ is uniformly bounded in time.
\end{itemize}
\end{theorem}
\begin{remark}[Comparison with Cucker-Smale model] In the original Cucker-Smale model, the interaction function is $\phi(r)=\frac{1}{(1+r^2)^\beta}$. For full connection
(i.e. when every node is connected with every other note, and thus $d^*=1$), one has that alignment occurs for $\beta\leq \frac12$, while flocking occurs for $\beta<1/2$. In our case, one has that alignment occurs when $\beta\leq\frac1{2d^*}$, while flocking occurs for $\beta<\frac1{2d^*}$. Also in this case, we see that the length $d^*$ of the $(T,\mu)$-connectivity graph plays a central role in the long-time behavior of the system. We will provide examples of this phenomenon in Section \ref{s-examples}.
\end{remark}

The structure of the article is as follows: in Section \ref{s-cooperative} we collect main properties of cooperative systems, i.e. of the form \eqref{e-ODE}, and provide some key examples of models of opinion formation. In Section \ref{s-main} we state the main theorem and prove it. In Section \ref{s-consequences} we provide some consequences and generalizations, i.e the case of Persistent Excitation, Integral Scrambling Coefficients, and Theorem \ref{t-main2} about second-order systems. In  Section \ref{s-examples} we show with some examples the relevance of our results, the fact that they are sharp, and compare them with other conditions of convergence. We draw some conclusions in Section \ref{s-conclusions}.

\section{Cooperative systems}
\label{s-cooperative}

In this section, we provide main results about cooperative systems. After recalling well-posedness of the solution, we  prove results about contractivity of the support. We then prove that we can restrict ourselves to the study of one-dimensional linear systems. We finally discuss several different models of cooperative systems available in the literature.

\subsection{Well-posedness and contractivity of the support} \label{s-contractive}

In this article, we study cooperative first-order systems of the form \eqref{e-ODE} and second-order systems of the form \eqref{e-ODE-2nd}. We always assume conditions (H1)-(H2) stated in Corollary~\ref{c-main} for systems of the form \eqref{e-ODE},
and conditions (H1)-(H2') for systems of the form \eqref{e-ODE-2nd}. These hypotheses ensure existence, globally in time, and uniqueness for the solution to the associated Cauchy problem, see e.g.~\cite{filippov}. Solutions are considered in the Carathéodory sense for the rest of the article: trajectories are absolutely continuous functions and  \eqref{e-ODE} or  \eqref{e-ODE-2nd} holds for almost every time. 
 General results about cooperative systems can also be found in~\cite{smith}.
 
We now provide a key property of cooperative systems: the support of the solution is (weakly) contractive.\begin{proposition} \label{p-contractive} Let $x(t)$ be a solution of~\eqref{e-ODE}. Define the support of the solution at time $t$ as
\begin{equation}\label{e-support}
\mathrm{supp}(x(t)):=\mathrm{conv}(\{x_i(t)\}),
\end{equation} i.e.~the (closed) convex hull of the set of $x_i$ at time $t$. Then, for $0\leq t\leq s$ it holds $\mathrm{supp}(x(t))\supseteq \mathrm{supp}(x(s))$. 

In dimension $d=1$, this implies that the maximum function $x_+(t):=\max_j\{x_j(t)\}$ is non-increasing and the minimum function $x_-(t):=\min_j\{x_j(t)\}$ is non-decreasing.

Define the diameter of the solution as the diameter of the support, i.e.
$$\diam{t}:=\max_{i,j}|x_i(t)-x_j(t)|.$$
The previous result implies that the diameter is a nonincreasing function.
\end{proposition}
\begin{proof} First observe that $\mathrm{supp}(x(t))$ is the convex hull of a finite number of points, hence it is a closed polygon.

Let $t$ be a time in which $x(t)$ is differentiable. If $x_i(t)$ belongs to the interior of $\mathrm{supp}(x(t))$, by continuity it belongs to the interior of $\mathrm{supp}(x(t+h))$ for $h>0$ sufficiently small. Assume then that $x_i(t)$ belongs to the boundary of $\mathrm{supp}(x(t))$: each term 
$$\frac{\lambda_i(x)}N M_{ij}(t)\phi_{ij}(x_i,x_j)(x_j-x_i)$$
points inwards in the polygon, eventually being tangent: this is due to the fact that $x_j$ belongs to the polygon and $\lambda_i,M_{ij},\phi_{ij}$ are positive. Then, the sum of all terms, that is $\dot x_i$, points inwards. Thus, one has $x_i(t+h)\in \mathrm{supp}(x(t))$ for $h>0$ sufficiently small. By merging the two cases, one has $x_i(t+h)\in \mathrm{supp}(x(t))$ for all $i=1,\ldots,N$, hence by convexity $\mathrm{supp}(x(t+h))\subseteq \mathrm{supp}(x(t))$. This proves the first result.

The results in dimension $d=1$ directly follows, since $\mathrm{supp}(x(t))$ is an interval. The statement about the diameter is a direct consequence too. \end{proof}

We strengthen the previous statement in dimension $d=1$, as follows, in the case of constant extremal values.

\begin{lemma}\label{L:belowmaximumpreserved}
Consider a trajectory $x(t)$ of~\eqref{e-ODE} in $\R$ such that $x_{+}^{*}=\max\{x_{i}(t)\ :\ i=1,\dots,N\}$ is constant.
Then the set $I^+(t)$ of indices $i$ that realize this maximum is non-increasing in time: if $i\notin I^+(t)$ then $i\notin I^+(t+h)$ for all $h>0$.

Similarly, assume that $x_{-}^{*}=\min\{x_{i}(t)\ :\ i=1,\dots,N\}$ is constant.
Then the set $I^-(t)$ of indices $i$ that realize this minimum is non-increasing in time.
\end{lemma}

\begin{proof} 
Consider an index $i\notin I^+(T)$ for some $T\geq 0$, which means $x_i(T)<x^*_+$. Define $f(t):=x^*_+-x_i(t)$, that satisfies $f(T)>0$. Let $t$ be a point in which $x(t)$ is differentiable. By the dynamic~\eqref{e-ODE} it holds
\begin{align*}
\dot f(t)&=0-\frac{\lambda_i(x)}N \sum_{j=1}^{N} M_{ij}(t)\phi_{ij}(x_i,x_j)(x_j-x_i)\geq -\frac{\lambda_i(x)}N \sum_{j=1}^{N} M_{ij}(t)\phi_{ij}(x_i,x_j)\left(x^*_+-x_i\right)\\
&=-\psi(t)f(t),
\end{align*}
where $\psi(t):=\frac{\lambda_i(x(t))}N \sum_{j=1}^{N} M_{ij}(t)\phi_{ij}(x_i(t),x_j(t))\geq 0$. In the first inequality we used that $x_{j}\leq x^{*}_{+}$, for all $j=1\,,\dots\,,N$, since $x^{*}_{+}$ is a maximum.
Gronwall lemma now ensures $$f(t)\geq f(T)\cdot\exp\left(-\int_T^t \psi(s)\,ds\right)>0.$$
By continuity, the estimate holds for all $t\geq T$, ensuring that $i\not \in I^+(t)$ for all $t\geq T$.

The proof for the second statement is equivalent.
\end{proof}

\subsection{Reduction to 1-dimensional linear systems}\label{s-reduction}
 In our article, we study convergence to consensus for~\eqref{e-ODE} by considering all possible connection functions $M_{ij}(t)$. As a consequence, it is not restrictive to assume that the dynamics is linear. Indeed, we have the following simple results.

 \begin{proposition}\label{p:non-lin} 
 Consider a solution $x(t)$ to \eqref{e-ODE} satisfying hypotheses (H1)-(H2) in Corollary \ref{c-main} and define the associated $(T,\mu)$-connectivity graph $G(t)$. Define $\overline{m}$ as in \eqref{e-om}. Then, there exist $\mathscr{L}^1$-integrable functions $\widetilde{\ctr{i}{j}}:[0,+\infty)\to[0,\overline{m}]$ such that $x(t)$ solves the linear system
 \begin{equation}
\label{e-ODElin}
	\dot x_{i}=\frac{1}{N}\sum_{j=1}^{N} \widetilde{M_{ij}}(t)\left(x_{j}-x_{i}\right),\qquad \ i=1,\dots,N\,.
 \end{equation}
Moreover, the $(T,\underline{m} \mu)$-connectivity graph associated to \eqref{e-ODElin} $\tilde G(t)$, with $\underline{m}$ given by \eqref{e-um}, contains $G(t)$.

Then, convergence to consensus of all trajectories of \eqref{e-ODE} is a consequence of the same result for \eqref{e-ODElin}.
\end{proposition}
\begin{proof}
 Consider a trajectory $x(t)$ of~\eqref{e-ODE} and assume that $t$ is a time for which $x$ is differentiable. It holds $$\dot x_i= \frac{1}{N}\sum_{j=1}^{N} \widetilde{M_{ij}}(t)\left(x_{j}-x_{i}\right),$$
 by choosing $$\widetilde{M_{ij}}(t):=\lambda_i(x) M_{ij}(t)\phi_{ij}(x_i,x_j)\,.
 $$
 The new coefficients $ \widetilde{M_{ij}}$ are integrable. Indeed, Proposition \ref{p-contractive} ensures $|x(t)|\leq C_1:=|x(0)|$, hence $$\lambda_i(x)\phi_{ij}(x_j(t),x_i(t))\leq \overline{m}.$$
 It then holds 
 $$  \widetilde{M_{ij}}(t)\leq \overline{m} {M_{ij}}(t).$$
 By observing that the definition of $\underline{m}$ in \eqref{e-um} implies 
 $$\lambda_i(x)\phi_{ij}(x_j(t),x_i(t))\geq \underline{m},$$
 it also holds $$\frac1T \int_t^{T+t} \tilde M_{ij}(s)\, ds\geq \frac1T \int_t^{T+t} \underline{m} M_{ij}(s)\, ds\geq \underline{m} \mu.$$
 This proves the result about $\tilde G(t)$ and  $G(t)$. 
 \end{proof}

Thanks to this simple result, from now on we will only consider the linear dynamics given in~\eqref{e-ODElin}. 
We even restrict ourselves to 1-dimensional systems. Moreover, we choose the maximum value of the connection functions $M_{ij}$ as $C=1$. This is the meaning of the following result.

\begin{proposition} \label{p-multid}
 Let $d\in\mathbb{N}$ and $v\in \R^d$. Consider a trajectory $$x(t)=(x_1(t),\ldots,x_N(t))$$ of the system \eqref{e-ODElin}, starting from a fixed initial condition $(x_1(0),\ldots, x_N(0))$ and with connection functions $\widetilde{M_{ij}}(t)$. Denote with $\tilde G(t)$ the corresponding $(T/\overline{m},\mu)$-connectivity graph, where $\overline{m}$ is given by \eqref{e-om}. Define the projected and time-rescaled trajectory $$y(t;v)=(y_1(t),\ldots,y_N(t))$$ with $y_i(t)\in \R$ defined by $y_i(t):=x_i(t/\overline{m})\cdot v$. Then, $y(t;v)$ is the unique solution to 
 \begin{eqnarray}\label{e-ODE1}
 \dot y_i=\frac{1}{N}\sum_{j=1}^N \hat{M}_{ij}(t) (y_j(t)-y_i(t)),
 \end{eqnarray}  with projected initial data $y_i(0):=x_i(0)\cdot v$ and the connection functions $\hat M_{ij}(t):=\frac{1}{\overline{m}}\widetilde{M_{ij}}(t/\overline{m})$. Then, the associated $(T,\mu/\overline{m})$-connectivity graph $\hat G(t)$ contains $\tilde G(t/\overline{m})$.
\end{proposition}
\begin{proof} We prove the first statement. Let $x(t)$ be a trajectory. At times $t$ for which $x$ is differentiable, by differentiating the identity $y_i(t)=x_i(t/\overline{m})\cdot v$, we have
\begin{eqnarray}
\dot y_i(t) &= &\frac{1}{N}\sum_{j=1}^N\frac{\widetilde{M_{ij}}(t/\overline{m})}{\overline{m}}(x_j(t/\overline{m})-x_i(t/\overline{m}))\cdot v=\frac{1}{N}\sum_{j=1}^N \hat M_{ij}(t)(y_j(t)-y_i(t)).\label{e:connectionproj}
\end{eqnarray}

We now prove the result about $\hat G(t)$. It holds
\begin{eqnarray*}
&&\frac{1}{T} \int_t^{T+t} \hat M_{ij}(s)\,ds=\frac{1}{T\overline{m}} \int_t^{T+t} \widetilde M_{ij}(s/\overline{m})\,ds=\frac{1}{T} \int_{t/\overline{m}}^{(T+t)/\overline{m}} \widetilde M_{ij}(s')\,ds'=\\
&&\frac{1/\overline{m}}{T/\overline{m}} \int_{t/\overline{m}}^{t/\overline{m}+T/\overline{m}} \widetilde M_{ij}(s')\,ds'\geq \mu/\overline{m}.
\end{eqnarray*}
\end{proof}

Thanks to these propositions, from now on we will mainly prove our results for the first-order 
$1$-dimensional linear system~\eqref{e-ODE1}.

\subsection{Models of opinion formation}
\label{s-models}
In this section, we discuss some models of opinion formation, of the form \eqref{e-ODE-noM}. Indeed, self-organization and emergence of patterns derived from collective dynamics have been widely studied in the literature, and many models have been proposed in order to account for more complex behavior or to catch some specific properties. In particular, key differences that we want to highlight are encoded in the choice of the scaling parameters $\lambda_i$ in \eqref{e-ODE-noM}.

In the classical case for \eqref{e-ODE-noM}, the function $\phi(x_i,x_j)$ is symmetric and $\lambda_i > 0$ is fixed for all $i \in \{1,\ldots,N\}$, see \cite{rainer2002opinion,degroot1974reaching}. A key property that facilitates the analysis of this model is the symmetry of the influence function with respect to $i$ and $j$. For general symmetric influence functions, the mean value is preserved. The system is cooperative, and if the influence function remains uniformly bounded below by a strictly positive constant over time, the dynamics converge to the initial mean state. This symmetry also enables the use of spectral methods and $l^2$-based techniques to study the system’s variance. Anyway, in the case of lack of connection that we address in this article, symmetry is broken if we do not require the additional constraint $M_{ij}=M_{ji}$.

A limitation of the classical model with constant weightings is that each agent’s dynamics is influenced by the collective weight of all agents. Consequently, under such scaling, the internal interactions of a small group located far from a much larger cluster become negligible. To address this, and originally intended for second-order models, in \cite{motsch2011new} a more realistic formulation has been proposed with a modified weighting procedure:
\begin{align*}
    \lambda_i = \frac{N}{\sum_{l=1}^N \phi(x_i,x_l)}.
\end{align*}
In this setting, the influence of agent $j$ on agent $i$ is normalized by the total influence $\sum_{l=1}^N \phi(x_i,x_l)$ acting on agent $i$. As a drawback, even when basic interactions $\phi(x_i,x_l)$ are symmetric, the resulting model is not symmetric anymore.

We highlight that our results in this article cover all choices of weights $\lambda_i$, that are somehow included in the definitions of $\underline{m},\overline{m}$ given in \eqref{e-um}-\eqref{e-om}.

\section{Proof of main result}
\label{s-main}

In this section, we prove the main result of this article, i.e. Theorem \ref{t-main}. As already stated above, we first  consider the first-order 1-dimensional linear case
\begin{equation}
\label{e-ODEmain}
\dot x_i(t)=\frac{1}{N}\sum_{j=1}^N M_{ij}(t) (x_j-x_i)\mbox{~~with~~}M_{ij}:[0,+\infty)\mapsto [0,1],
 \end{equation}
 that covers all the generalizations, thanks to Proposition \ref{p-multid}. Indeed, when we rewrite~\eqref{e-ODE} as a linear system, as in Proposition \ref{p:non-lin}, the $(T,\mu)$-connectivity graph $G(t)$ is transformed into the $(T,\underline{m}\mu)$-connectivity graph $\tilde G(t)$. On the other hand, applying Proposition \ref{p-multid} the graph $\tilde G(t)$ is transformed into the $(T\overline{m},\underline{m}\mu/\overline{m})$-connectivity graph $\hat G(t\overline{m})$ for a system of the form~\eqref{e-ODEmain}.
 
 The proof is based on the following steps: in Section \ref{s-barrier}, we introduce some barrier functions that allow to provide estimates about solutions of \eqref{e-ODEmain}. In Section \ref{s-graph}, we provide estimates about the solutions based on the connectivity graph. In Section \ref{s-proof}, we finally give the proof of Theorem \ref{t-main} and Corollary \ref{c-main} in full generality.
 
 \subsection{Barrier functions}
 \label{s-barrier}
 
  
 In this section, we provide some estimates about solutions of \eqref{e-ODEmain}, based on some exponential functions.

\begin{proposition}[Left barrier]\label{p-barrier} Define $$\Psi_L(\bar x,\alpha,t):=\alpha+\exp\left(-\tfrac{N-1}N t\right)(\bar x-\alpha).$$
Let $x(t)$ be a solution of \eqref{e-ODEmain} with $x_i(0)\geq \alpha$ for all $i\in\{1,\ldots,N\}$,
and let $\bar x \geq\alpha$. If there exists an index $I$ and a time $\tau\in[0,T]$ such that $x_I(\tau)\geq \Psi_L(\bar x,\alpha,\tau)$, then for all $t\geq \tau$ it also holds $x_I(t)\geq \Psi_L(\bar x,\alpha,t)$.
\end{proposition}
\begin{proof} We first observe that $x_i(0)\geq \alpha$ for all $i=1,\ldots,N$ implies $x_i(t)\geq \alpha$ for all $t\geq 0$, due to Proposition \ref{p-contractive}. Consider the function $f(t):=x_I(t)-\Psi_L(\bar x,\alpha,t)$ and,  for $t\geq \tau$, compute
\begin{eqnarray}
\frac{d}{dt} f(t)&=&\frac1N \sum_{x_j(t) \leq \Psi_L(\bar x,\alpha,t),j\neq I} M_{Ij}(t)(x_j(t)-x_I(t))+\frac1N \sum_{x_j(t)> \Psi_L(\bar x,\alpha,t),j\neq I} M_{Ij}(t)(x_j(t)-x_I(t))+\nonumber \\
&& \tfrac{N-1}N \exp\left(-\tfrac{N-1}N t\right)(\bar x-\alpha)\nonumber \\
&\geq& \frac1N \sum_{x_j\leq \Psi_L(\bar x,\alpha,t),j\neq I} M_{Ij}(t)\bigg( \Big(\alpha-\Psi_L(\bar x,\alpha,t) \Big)   + \Big(\Psi_L(\bar x,\alpha,t)-x_I(t) \Big) \bigg) +\nonumber\\
&& \hspace{-5mm} \frac1N \sum_{x_j(t)> \Psi_L(\bar x,\alpha,t),j\neq I} M_{Ij}(t) \Big(\Psi_L(\bar x,\alpha,t)-x_I(t)\Big)  + \tfrac{N-1}N \exp\left(-\tfrac{N-1}N t\right)(\bar x-\alpha)\nonumber\\
&=& \frac1N \sum_{x_j\leq \Psi_L(\bar x,\alpha,t),j\neq I} M_{Ij}(t)\Big(\alpha-\Psi_L(\bar x,\alpha,t) \Big) + \frac1N \sum_{j\neq I} M_{Ij}(t)(\Psi_L(\bar x,\alpha,t)-x_I(t))+ \nonumber \\
&& \tfrac{N-1}N \exp\left(-\tfrac{N-1}N t\right)(\bar x-\alpha)\nonumber \\
&\geq &\frac1N \sum_{j\neq I} M_{Ij}(t) \Big(-  \exp\left(-\tfrac{N-1}N t\right)  (\bar x-\alpha) \Big)  - \bigg(\frac1N \sum_{j\neq I} M_{Ij}(t) \bigg) f(t) + \tfrac{N-1}N \exp\left(-\tfrac{N-1}N t\right)(\bar x-\alpha)\nonumber\\
&= &\frac1N \sum_{j\neq I} (1-M_{Ij}(t)) \exp\left(-\tfrac{N-1}N t\right)(\bar x-\alpha) 
 - \bigg(\frac1N \sum_{j\neq I} M_{Ij}(t) \bigg) f(t) \geq 0-r(t) f(t).
\label{e-stimaf}
\end{eqnarray}
In the last inequality, we used $M_{Ij}(t) \in[0,1]$ for $\mathcal{L}^1$-almost every $t \geq 0$ and defined $r(t):=\frac1N \sum_{j\neq I} M_{Ij}(t)$, that is independent on $f(t)$ and satisfies $r(t)\leq \frac{N-1}N$ for $\mathcal{L}^1$-almost every $t \geq 0$.
A direct application of Gr\"onwall's lemma at time $\tau$ ensures that
\begin{equation*}
f(t)\geq \exp\left(-\int_\tau^t r(s)\,ds\right) f(\tau).
\end{equation*}
Our hypothesis now reads as $f(\tau)\geq 0$, which implies that $f(t)\geq 0$ for $t\geq \tau$. This proves the result.
\end{proof}

\subsection{Estimates and connectivity graph}
\label{s-graph}

In this section, we provide estimates about trajectories of \eqref{e-ODEmain}, based on connections of the connectivity graph.

We first prove an estimate about the distance between the trajectories of two agents $x_i(t),x_j(t)$, in the case the arrow $i\to j$ exists in the connectivity graph.
\begin{proposition}\label{p-1step} Let $T,\mu>0$ be given. Let $x(t)$ be a solution of \eqref{e-ODEmain} with $x_k(0)\geq \alpha$ for all $k\in\{1,\ldots,N\}$. Assume that the $(T,\mu)$-connectivity graph $G(t)$ contains the directed edge $i\to j$. It then holds
\begin{eqnarray}
\label{e-1step}
x_i(T)\geq \alpha+\eta\exp(-2\tfrac{N-1}{N}T)(x_j(0)-\alpha)
\end{eqnarray}
with $\eta=\frac{\mu T}{N+\mu T}$.
\end{proposition}
\begin{proof}
We first define the following barrier: $$\gamma_j(t):=\Psi_L(x_j(0),\alpha,t)=\alpha+\exp(-\tfrac{N-1}{N}t)(x_j(0)-\alpha).$$
As a first case, assume that $x_i(0)\geq \gamma_j(T)$. We rewrite it as $x_i(0)\geq \Psi_L(\gamma_j(T),\alpha,0)$  and apply Proposition \ref{p-barrier}. Remark that we can apply such proposition
to the left barrier $\gamma_j$, since $\gamma_j(T)\geq \alpha$ as a consequence of $x_j(0)\geq\alpha$. It then holds
\begin{eqnarray*}
x_i(T)&\geq& \Psi_L(\gamma_j(T),\alpha,T)\geq \alpha+\exp(-\tfrac{N-1}{N}T)(\alpha+\exp(-\tfrac{N-1}{N}T)(x_j(0)-\alpha)-\alpha)=\\
&&\alpha+\exp(-2\tfrac{N-1}{N}T)(x_j(0)-\alpha).
\end{eqnarray*}
This condition is stronger than the statement, since $1>\eta$.

We then assume that it holds $x_i(0)<\gamma_j(T)$ from now on. We define 
$$l:=\left(1+\frac{\mu T}N\right)^{-1}(\gamma_j(T)-x_i(0)),$$
which is strictly positive by hypothesis. We now define a second barrier function, that is
\begin{eqnarray*}
&&\gamma_i(t):=\Psi_L(\gamma_j(T)-l,\alpha,t)=\alpha+\exp(-\tfrac{N-1}{N}t)(\gamma_j(T)-l-\alpha)=\\
&&\alpha+\exp(-\tfrac{N-1}{N}(t+T))(x_j(0)-\alpha)-\exp(-\tfrac{N-1}{N}t)l.
\end{eqnarray*}

We have two cases: either $x_i(t)\geq \gamma_i(t)$ for some $t\in[0,T]$ or $x_i(t)< \gamma_i(t)$ for all $t\in[0,T]$. In the first case, by the definition of $l$ and recalling that $x_i(0),x_j(0)\geq\alpha$, we first observe the following estimate:
\begin{eqnarray*}
&&\exp(-\tfrac{N-1}{N}T)(x_j(0)-\alpha)\geq \left(1+\frac{\mu T}N\right)^{-1}\exp(-\tfrac{N-1}{N}T)(x_j(0)-\alpha) \geq l.
\end{eqnarray*}
This implies 
$\gamma_j(T)-l\geq \alpha$. We can then apply Proposition~\ref{p-barrier} to the left barrier $\gamma_i$, that ensures $x_i(T)\geq\gamma_i(T)$. This reads as
\begin{eqnarray*}
x_i(T)&\geq& \alpha+\exp(-2\tfrac{N-1}{N}T)(x_j(0)-\alpha)\\
&&-\exp(-\tfrac{N-1}{N}T)\left(1+\frac{\mu T}N\right)^{-1}(\alpha+\exp(-\tfrac{N-1}{N}T)(x_j(0)-\alpha)-x_i(0))\\
&=& \alpha+\exp(-2\tfrac{N-1}{N}T)\eta (x_j(0)-\alpha)+\exp(-\tfrac{N-1}{N}T)\left(1+\frac{\mu T}N\right)^{-1}(x_i(0)-\alpha)\\
&\geq& \alpha+\exp(-2\tfrac{N-1}{N}T)\eta (x_j(0)-\alpha).
\end{eqnarray*}
Here, we used that $x_i(0)-\alpha\geq 0$ by hypothesis on the support, and that $\eta=1-\left(1+\frac{\mu T}N\right)^{-1}$. Also in this case, the statement is proved.

We are then left with the case $x_i(t)< \gamma_i(t)$ for all $t\in[0,T]$. We prove that this case cannot occur, by contradition. We first observe that, due to the fact that left barriers are decreasing functions of time, for all $t\in[0,T]$ it holds both $x_i(t)< \gamma_j(T)-l$ and $x_j(t)\geq \gamma_j(T)$, i.e. $x_j(t)-x_i(t)\geq l$. We now provide an estimate similar to \eqref{e-stimaf} with this added information. We define $f(t):=x_i(t)-\gamma_i(t)$, that is strictly positive, continuous and differentiable for almost every $t\in[0,T]$. We also observe that the following identity holds:
$$\dot \gamma_i(t)=-\tfrac{N-1}{N}\exp(-\tfrac{N-1}{N}t)(\gamma_i(0)-\alpha).$$

 For points of differentiability, it then holds
\begin{eqnarray}
\frac{d}{dt} f(t)&=&\frac1N M_{ij}(t)(x_j(t)-x_i(t))+\frac1N \sum_{x_k(t) \leq \gamma_i(t),k\neq i} M_{ik}(t)(x_k(t)-x_i(t))\nonumber \\
& &+\frac1N \sum_{k\not\in\{ i,j\}, x_k(t)> \gamma_i(t)} M_{ik}(t)(x_k(t)-x_i(t))+\tfrac{N-1}{N}\exp(-\tfrac{N-1}{N}t)(\gamma_i(0)-\alpha)\nonumber \\
&\geq& \frac1N M_{ij}(t)l +\frac1N \sum_{x_k(t) \leq \gamma_i(t),k\neq i} M_{ik}(t)((\alpha-\gamma_i(t))+(\gamma_i(t)-x_i(t)))\nonumber \\
& &+\frac1N \sum_{k\not\in\{i, j\}, x_k(t)> \gamma_i(t)} M_{ik}(t)(\gamma_i(t)-x_i(t))+\tfrac{N-1}{N}\exp(-\tfrac{N-1}{N}t)(\gamma_i(0)-\alpha)\nonumber \\
&\geq& \frac1N M_{ij}(t)l +\frac1N \sum_{x_k(t) \leq \gamma_i(t),k\neq i} M_{ik}(t)(\alpha-\gamma_i(t))+\frac1N \sum_{k\not\in\{i, j\}} M_{ik}(t)(\gamma_i(t)-x_i(t))+\nonumber\\
&&\tfrac{N-1}{N}\exp(-\tfrac{N-1}{N}t)(\gamma_i(0)-\alpha).\nonumber 
\end{eqnarray}
We now observe that it holds $$\alpha-\gamma_i(t)= -\exp(-\tfrac{N-1}{N}t)(\gamma_i(0)-\alpha).$$ We then have 
\begin{eqnarray}
\frac{d}{dt} f(t)&\geq& \frac1N M_{ij}(t)l +\frac1N \sum_{k\neq i} M_{ik}(t)(-\exp(-\tfrac{N-1}{N}t)(\gamma_i(0)-\alpha)) \nonumber\\
&&-\left(\frac1N \sum_{k\not\in\{i, j\}} M_{ik}(t)\right) f(t) +\tfrac{N-1}{N}\exp(-\tfrac{N-1}{N}t)(\gamma_i(0)-\alpha)\nonumber\\
&\geq&\frac1N M_{ij}(t)l + \frac1N \sum_{k\neq i} (1-M_{ik}(t))\exp(-\tfrac{N-1}{N}t)(\gamma_i(0)-\alpha)) -\left(\frac1N \sum_{k\neq i} M_{ik}(t)\right) f(t)\nonumber\\
&\geq& \frac1N M_{ij}(t)l +0- \frac{N-1}N f(t).
\label{e-stimaf2}
\end{eqnarray}
Again, here we used $M_{ik}(t)\in [0,1]$. A direct application of  Gr\"onwall's lemma now implies 
\begin{eqnarray*}
&&f(T)\geq \exp(-\tfrac{N-1}{N}T)\left[ f(0)+\int_0^T\exp(s)\frac1N M_{ij}(s) l\,ds\right]\\
&&\geq \exp(-\tfrac{N-1}{N}T) \left[x_i(0)-\gamma_i(0)+1\cdot \frac l N  \int_0^T
M_{ij}(s)\,ds\right].
\end{eqnarray*}
Since $i\to j$ is an arrow of $G(t)$, it holds $\frac1T \int_{0}^T M_{ij}(s)\,ds\geq \mu$. It then holds
\begin{eqnarray*}
x_i(T)-\gamma_i(T)&=&f(T)\geq \exp(-\tfrac{N-1}{N}T) \left(x_i(0)-\left(\gamma_j(T)-l\right)+ \frac{l\, T}N \mu\right)=\\
&&\exp(-\tfrac{N-1}{N}T) \left(x_i(0)-\gamma_j(T)+l\left(1+\frac{\mu\, T}N\right)\right)=0.
\end{eqnarray*}
This raises a contradiction with the condition $x_i(t)< \gamma_i(t)$ for all $t\in[0,T]$.\end{proof}

We now extend the previous result to a whole path from $i$ to $j$. We highlight that the extensions requires that the path is in the connectivity graph for several time intervals at distance $T$.
\begin{proposition}\label{p-keylemma}Let $T,\mu>0$ be given. Let $x(t)$ be a solution of \eqref{e-ODEmain} with $x_i(0)\geq \alpha$ for all $i\in\{1,\ldots,N\}$. Assume that the $(T,\mu)$-connectivity graph $G(t)$ contains the directed sequence ${i_0}\to {i_1}\to\ldots \to {i_k}$ for all times $t=0,T,\ldots,(k-1)T$. Then, for all $t\geq kT$ it holds
\begin{equation}
\label{e-frecursive}
x_{i_{l}}(t)\geq \alpha+\eta^k\exp(-2\tfrac{N-1}{N}t)(x_{i_k}(0)-\alpha)
\end{equation}
for all $l\in\{0,\ldots,k\}$, where $\alpha:=\min_{j\in\{1,\ldots,N\}} x_j(0)$ and $\eta:=\frac{\mu T}{N+\mu T}$.
\end{proposition}
\begin{proof} We first observe that 
$\eta\in(0,1)$. The case $l=k$ is then a consequence of applying Proposition \ref{p-barrier} for $k-1$ times. We then focus on $l\in\{0,\ldots,k-1\}$. For simplicity of notation, we use the index $\ell:=k-l$. 

 We first prove the following recursive version of Proposition \ref{p-1step}: for $\ell\in\{1,\ldots, k\}$ it holds
\begin{equation}\label{e-inductionstep}x_{i_{k-\ell}}(\ell T)\geq
\alpha+\eta^\ell \exp(-2\tfrac{N-1}{N}\ell T)(x_{i_k}(0)-\alpha).
\end{equation}
The proof is as follows: the case $\ell=1$ is exactly Proposition~\ref{p-1step} applied with 
$i=i_{k-1}$ and $j=i_k$, 
and observing that $G(0)$ contains the directed edge $i_{k-1}\to i_k
$. We now assume that the formula holds for $\ell\in\{1,\ldots,k-1\}$ and prove it for $\ell+1$. \\
We first apply Proposition \ref{p-1step} by choosing $x_j(0)=x_{i_{k-\ell}}(\ell\,T)$ and observing that $G(\ell\,T)$ contains the arrow $i_{k-(\ell+1)}\to i_{k-\ell}$ by construction. It then holds
\begin{equation}
\label{ind-proof1b}
x_{i_{k-(\ell+1)}}((\ell+1)T)\geq \alpha+\eta\exp(-2\tfrac{N-1}{N}T)(x_{i_{k-\ell}}(\ell\,T)-\alpha).
\end{equation}

Next, by applying the induction hypothesis \eqref{e-inductionstep}, we deduce
\begin{equation}
\label{ind-proof2}
\begin{aligned}
x_{i_{k-\ell}}(\ell T)&\geq \alpha+\eta\exp(-2\tfrac{N-1}{N}T)(
\alpha+\eta^{\ell-1}\exp(-2\tfrac{N-1}{N}(\ell-1)T)(x_{i_k}(0)-\alpha)
-\alpha)\\
&=\alpha+\eta^{\ell}\exp(-2\tfrac{N-1}{N}\ell T)(x_{i_k}(0)-\alpha).
\end{aligned}
\end{equation}
Combining~\eqref{ind-proof1b}, \eqref{ind-proof2}, we deduce
\begin{eqnarray*}
&&x_{i_{k-(\ell+1)}}((\ell +1)T)\geq\alpha+\eta^{\ell+1} \exp(-2\tfrac{N-1}{N}(\ell +1)T)(x_{i_k}(0)-\alpha).
\end{eqnarray*}
The recursive formula \eqref{e-inductionstep} is then proved.

We now need to estimate the evolution of $x_{i_{k-\ell}}$ on the time interval $[\ell T,t]$. It is sufficient to apply Proposition \ref{p-barrier}, first on the $\lceil\tfrac{t}{T}\rceil-\ell$ intervals of length $T$, then on the time interval $\left[\lceil\tfrac{t}{T}\rceil T,\,t\right]$. Since the solution already starts on the right of the left barrier at time $\ell T$, this gives
\begin{eqnarray}
&&x_{i_{k-\ell}}(t)\geq \Psi_L(x_{i_{k-\ell}}(\ell T),\alpha,t-\ell T)\nonumber\\
&&\geq \alpha+\exp(-\tfrac{N-1}{N}(t-\ell T))(\alpha+\eta^{\ell}\exp(-2\tfrac{N-1}{N}\ell T)(x_{i_k}(0)-\alpha)-\alpha)\nonumber\\
&&=\alpha+\eta^{\ell}\exp(-\tfrac{N-1}{N}(t+\ell T))(x_{i_k}(0)-\alpha)\geq
\alpha+\eta^k\exp(-2\tfrac{N-1}{N}t)(x_{i_k}(0)-\alpha).\label{e-xid}
\end{eqnarray}
Here, we just used $\ell T\leq t$ and $\eta\in(0,1)$. By writing $l=k-\ell$, one gets the result.
\end{proof}

\subsection{Proof of Theorem \ref{t-main} and Corollary \ref{c-main}} \label{s-proof}
We are now ready to prove Theorem \ref{t-main}.

\begin{proof} We first consider solution of \eqref{e-ODEmain}, i.e. of \eqref{e-ODE} in 1 dimension with linear dynamics. Denote the support of the solution at time $t=0$ as the interval $[\alpha,\beta]$. The diameter is then $\diam{0}=\beta-\alpha$.

Under the hypothesis of the theorem, there exists a graph $G^*$ such that for all $l\in\N$ the graph $G(lT)$ contains all its arrows. Moreover, $G^*$ admits a globally reachable node $I$. Since the length of the graph $G^*$ is $d^*:=d(G^*,I)$, for each index $i\in\{1,\ldots,N\}$ there exists a directed sequence $i\to\ldots\to I$ of length $k\leq d^*$. Then, by applying Proposition \ref{p-keylemma} and observing that $d^* T\geq kT$, it holds
\begin{equation}\label{e-generic}
x_i(d^* T)\geq  \alpha+\eta^k\exp(-2\tfrac{N-1}{N}d^*T)(x_{I}(0)-\alpha).
\end{equation}
We now have two cases:
\begin{itemize}
\item it holds $x_I(0)\geq \frac{\alpha+\beta}{2}$. Then, \eqref{e-generic} implies 
$$x_i(d^*T)\geq  \alpha+\eta^{d(G^*,I)}\exp(-2\tfrac{N-1}{N}d(G^*,I)T)\frac{\beta-\alpha}{2}.$$
The estimate holds for any index $i$, then the support satisfies 
\begin{eqnarray}\label{diam-est1}
\
&&\diam{d^*T}\leq \beta- \alpha-\eta^{d^*}\exp(-2\tfrac{N-1}{N}d^*T)\tfrac{\beta-\alpha}{2} =  C \diam{0},\label{e-stimaxi}
\end{eqnarray}
with $C:=1-\tfrac12 \eta^{d^*}\exp(-2\tfrac{N-1}{N}d^*T).$

\item it holds $x_I(0)<\frac{\alpha+\beta}2$. Then, one replaces all $x_i(t)$ with $y_i(t)=\alpha+\beta-x_i(t)$, that satisfy all the hypotheses of Theorem \ref{t-main}. Moreover, it holds $y_I(0)\geq \frac{\alpha+\beta}2$, thus the results of the previous item hold, i.e. \eqref{e-stimaxi} holds for the $y_i$ variables. By writing $x_i(t)=\alpha+\beta-y_i(t)$, we have that \eqref{e-stimaxi} holds for the $x_i$ too.
\end{itemize}
By merging the two cases, we have that \eqref{e-stimaxi} holds. By induction, for all $n\in\mathbb{N}$ it holds
\begin{eqnarray*}
&&\diam{nd^*T}\leq C^n \diam{0}.
\end{eqnarray*}
We now aim to estimate the evolution of the support on the time interval $[nd^*T,(n+1)d^*T]$. On one side, the support is non-increasing, then for all $t\in [nd^*T,(n+1)d^*T]$ it holds
\begin{eqnarray}\diam{t}\leq  C^n \diam{0}.\label{e-stimaL}
\end{eqnarray}
 On the other side, it is sufficient to observe that $$|\dot x_i|\leq\frac1N \sum_{j=1}^N 1\cdot |x_j(t)-x_i(t)|\leq \frac{N-1}N \diam{t}$$
to deduce that it holds $$\frac{d}{dt}\diam{t}\geq -\max_{i,j\in\{1,\ldots,N\}} |x_i(t)|+|x_j(t)|\geq -2\frac{N-1}N \diam{t}.$$
This clearly implies that for all $\tau\in[0,d^*T]$ (seen as reverse time from $(n+1)d^*T]$ it holds
\begin{eqnarray}&&\diam{(n+1)d^*T-\tau}\leq \exp\left(2\frac{N-1}N\tau\right) \diam{ (n+1)d^*T} \leq C^{n+1} \exp\left(2\frac{N-1}N\tau\right) \diam{0}.\label{e-stimaR}
\end{eqnarray}
By choosing the optimal available estimate, we have that \eqref{e-stimaL} holds for $t\in[nd^*T,(n+1)d^*T-\delta]$ with $\delta=-\frac{N}{2(N-1)}\log(C)$, while \eqref{e-stimaR} holds for $t\in [(n+1)d^*T-\delta,(n+1)d^*T]$.

The $n$-dimensional case is a direct consequence, by applying Proposition \ref{p-multid}, with no rescaling of time.
\end{proof}

We now prove Corollary \ref{c-main}.

\begin{proof} We first rewrite \eqref{e-ODE} as a linear system, as in Proposition \ref{p:non-lin}. The $(T,\mu)$-connectivity graph $G(t)$ is transformed into the $(T,\underline{m}\mu)$-connectivity graph $\tilde G(t)$. We then apply Proposition \ref{p-multid}: the system is now 1-dimensional, the graph $\tilde G(t)$ is transformed into the $(T\overline{m},\underline{m}\mu/\overline{m})$-connectivity graph $\hat G(t\overline{m})$ and time is rescaled from $t$ to $t\overline{m}$. For such system, Theorem \ref{t-main} holds, where in the definitions of $\eta,C,\delta$ one needs to replace $T$ with $T\overline{m}$ and $\mu$ with $\underline{m}\mu/\overline{m}$. We then translate \eqref{e-stimaR} back to the original system \eqref{e-ODE}, by rescaling time and find the result.
\end{proof}

\section{Additional results}
\label{s-consequences}

In this section, we provide some consequences of our main Theorem \ref{t-main}. This also allows to compare our results with similar contributions in the literature.

\subsection{Persistent excitation}

In this section, we focus on the concept of Persistent Excitation, introduced and studied in \cite{narendra2012stable,ChSi2010,ChSi2014,ren2008distributed,tang2020bearing,manfredi2016criterion,bonnet2021consensus,anderson2016convergence,chaillet2008uniform}. We recall its definition.

\begin{definition}[Persistent Excitation] Let $T,\mu>0$. The connection functions $M_{ij}(t)$ satisfy the Persistent Excitation condition with parameters $(T,\mu)$ if for all indexes $i,j$ and times $t\geq 0$ it holds $\frac1T \int_t^{t+T}M_{ij}(s)\,ds\geq \mu$.
\end{definition}

It is clear that the condition can be restated as follows: the $(T,\mu)$-connectivity graph given in Definition \ref{d-conngraph} is complete, i.e. for all pair of distinct indexes $i,j$ there exists an edge $i\to j$ (and $j\to i$). As a consequence, the hypotheses of Theorem \ref{t-main} hold: the graph $G^*$ is the complete graph, thus all nodes are globally reachable and the length is $d^*=1$. The following corollary is then directly proved.

\begin{corollary} Let the system \eqref{e-ODE} be given and assume hypotheses (H1)-(H2) from Corollary \ref{c-main} hold. Assume moreover that the Persistent Excitation condition holds for some $(T,\mu)$. It then holds 
$$\max_{i,j}|x_i(nT)-x_j(nT)|\leq C^n \max_{i,j}|x_i(0)-x_j(0)|$$
with 
$$C=1-\frac12 \left( \tfrac{\underline{m}\mu T}{N+\underline{m}\mu T}    \right)\exp(-2\tfrac{N-1}{N}T\overline{m}),$$
and $\underline{m},\overline{m}$ are given by \eqref{e-um}-\eqref{e-om}.

\end{corollary}

\subsection{Integral Scrambling Coefficients} \label{s-ISC}
In this section, we focus on the concept of Integral Scrambling Coefficients (ISC from now on), introduced and studied in \cite{bonnet2023consensus,bonnet2022consensus}. We recall its definition.

\begin{definition}[Integral Scrambling Coefficients] Let $T,\mu>0$. The connection functions $M_{ij}(t)$ satisfy the Integral Scrambling Coefficients condition with parameters $(T,\mu)$ if for all indexes $i,j$ and times $t\geq 0$ there exists an index $k$ such that both the following conditions hold: $$\frac1T \int_t^{t+T}M_{ik}(s)\,ds\geq \mu,\qquad
\frac1T \int_t^{t+T}M_{jk}(s)\,ds\geq \mu.$$
\end{definition}
It is clear that this condition is a generalization of the Persistent Excitation condition, as it is satisfied by choosing $k=i,j$. Instead, it is remarkable to observe that the original Moreau Condition is not a generalization of the ISC condition for the following reason: the index $k$ depends on $i,j$, but also on time $t$, i.e. the  $(T,\mu)$-connectivity graph is not constant with respect to time. The goal of this section is then to study the $(T,\mu)$-connectivity graphs given by ISC condition and to build another connectivity graph with different parameters $T',\mu'$ that satisfies hypotheses of Theorem \ref{t-main}. Our result is the following.

\begin{corollary} Let the system \eqref{e-ODE} be given and assume hypotheses (H1)-(H2) from Corollary \ref{c-main} hold. Assume moreover that the Integral Scrambling Coefficients condition holds for some $(T,\mu)$. It then holds 
\begin{eqnarray}
&&\hspace{-5mm}\max_{i,j}|x_i(n(N-1)2^{N(N-1)}T)-x_j(n(N-1)2^{N(N-1)}T)| \leq C^n \max_{i,j}|x_i(0)-x_j(0)|, \label{est-ISC}
\end{eqnarray}
with 
$$C=1-\left( \tfrac{\underline{m}\mu T}{N+\underline{m}\mu T}    \right)^{N-1}\exp(-2\tfrac{N-1}{N}(N-1)^{(2^{N(N-1)})}T\overline{m}),$$
where $\underline{m},\overline{m}$ are given by \eqref{e-um}-\eqref{e-om}.
\end{corollary}
\begin{proof} Let $T,\mu$ be fixed. For each time $t\geq 0$, consider $G(t)$ the corresponding $(T,\mu)$-connectivity graph. The ISC condition ensures the following: for each pair $i,j$ there exists $k$ such that both $i\to k$ and $j\to k$ are contained in $G(t)$. It is easy to observe that this condition ensures that $G(t)$ admits a globally reachable node. Indeed, for a given $t\geq 0$ first denote with $\Gamma(i,j):=k$ the index\footnote{If more than one index $k$ satisfies the condition, it is sufficient to choose one among them.} given by the ISC condition. Then, extend the operator to finite sequences of length larger than 2:
\begin{itemize}
\item for  even $2n$-tuples
$$\Gamma(i_1,\ldots,i_{2n}):=\Gamma\big(\Gamma(i_1,i_2),\Gamma(i_3,i_4),\ldots,\Gamma(i_{2n-1},i_{2n})\big);$$
\item for  odd $2n+1$-tuples
$$\Gamma(i_1,\ldots,i_{2n+1}):=\Gamma\big(\Gamma(i_1,\ldots,i_{2n}),i_{2n+1}\big).$$
\end{itemize}

For simplicity, also define $\Gamma(i)=i$ for a single element. The image of such operator is an $n$-tuple with length strictly smaller than the one in the domain of $\Gamma$, except if the $n$-tuple consists of a single element. As a consequence, there exists a single element $I\in\{1,\ldots,N\}$ that satisfies 
$$I=\Gamma(I)=\Gamma^{N-1}(1,\ldots,N),$$
where $\Gamma^{N-1}$ is the composition of $\Gamma$ for $N-1$ times. By definition of the $(T,\mu)$-connectivity graph, this implies that $G(t)$ admits a globally reachable node, that is indeed $I$. The length $d(G(t),I)$ is clearly bounded from above by $N-1$ (as for any graph with a globally reachable node). One can moreover observe that $G(t)$ is a simple graph, i.e. for each pair of indexes $i,j$ there exists either zero or one edge $i\to j$.

We are now ready to define a (finite) sequence of $T_k,\mu_k>0$ and a corresponding $(T_k,\mu_k)$-connectivity graphs $G_k(t)$ for which Corollary \ref{c-main} applies. First reduce the study to the $(T,\mu)$-connectivity graphs $G(nT)$ with $n=0,\ldots,(N-1)2^N-1$. For future use, we define $(T_1,\mu_1):=(T,\mu)$ and denote the $(T_1,\mu_1)$-connectivity graph with $G_1(t)=G(t)$. Observe that each of the $G_1(nT)$ graphs belongs to the finite set of simple graphs with $N$ nodes. Notice that the cardinality of the set of simple graphs with $N$ nodes is bounded by  $2^{N(N-1)}$, indeed, they can be identified with a $N\times N$ matrix with $a_{ij}=1$ if there is an edge $i\to j$, and $a_{ij}=0$ otherwise. Notice that it always holds $a_{ii}=0$, then only $N(N-1)$ values have to be chosen in the set $\{0,1\}$. We denote simple graphs with $H^1,\ldots, H^{2^{N(N-1)}}$ for simplicity. We have two cases:
\begin{enumerate}
\item there exists a $m\in \{0,\ldots, 2^{N(N-1)}-1\}$ such that for all $r\in\{0,\ldots,N-2\}$ it holds $G_1((m(N-1)+r)T_1)=H^1$, i.e. there exists a sequence of $N-1$ consecutive $(T_1,\mu_1)$-connectivity graphs $G_1(nT_1)$ that coincide (and the value is $H^1$). One can then prove the estimate~\eqref{est-ISC} as follows:
\begin{itemize}
\item on the time interval $[0,m(N-1)T_1]$ one recalls that the support is contractive, hence 
$$\diam{m(N-1)T}\leq \diam{0};$$
\item on the time interval $$[m(N-1)T_1,(m+1)(N-1)T_1]$$ one applies Corollary~\ref{c-main}, with $G^*=H^1$ that has length $d^*\leq N-1$ and gets
\begin{eqnarray*}
\diam{(m+1)(N-1)T_1}\leq C_1\diam{m(N-1)T_1},
\end{eqnarray*}
with $C_1,\delta_1$ associated to $(T_1,\mu_1)$, i.e.  
\begin{equation*}
\begin{aligned}
C_1 &= 1-\tfrac12\left( \tfrac{\underline{m}\,\mu_1 T_1}{N+\underline{m}\,\mu_1 T_1}    \right)^{N-1}\exp(-2\tfrac{(N-1)^2}{N}T_1 \underline{m}),
\qquad
\delta_1 &= -\frac{\log(C_1)}{(N-1)T_1\, \underline{m}},
\end{aligned}
\end{equation*}
and  $d^*=N-1$. We then have
\begin{eqnarray*}
&&\diam{(m+1)(N-1)T_1}\leq \exp(-\delta_1 (N-1)T_1)\, \diam{m(N-1)T_1}.
\end{eqnarray*}

\item on the time interval $$[(m+1)(N-1)T_1,(N-1)2^NT]$$ one again uses the fact that the support is contractive, thus 
\begin{eqnarray}
&&\diam{(N-1)2^NT}\leq \diam{(m+1)(N-1)T_1}\leq C_1 \diam{0}.\label{e-suppISC}
\end{eqnarray}
Remark that the final time does not depend on $T_1$.
\end{itemize}
\item the previous condition does not hold, i.e. for all $m\in \{0,\ldots, 2^{N(N-1)}-1\}$ there exists $r\in\{0,\ldots,N-2\}$ such that $G_1((m(N-1)+r)T_1)\in \{H^2,\ldots,H^{2^{N(N-1)}}\}$. We now replace $T_1$ with $T_2:=(N-1)T_1$ and $\mu_2:=\frac{\mu_1}{N-1}$. We now consider the associated $(T_2,\mu_2)$-connectivity graphs, that we denote with $G_2(t)$. The previous condition now ensures that for all $m\in \{0,\ldots, 2^{N(N-1)}-1\}$ it holds $G_2(m(N-1)T_2)\in \{H^2,\ldots,H^{2^{N(N-1)}}\}$. Observe that we have removed one element from the set of possible graphs. We can then iterate the process:
\begin{itemize}
\item either there exists a $m\in \{0,\ldots, 2^{N(N-1)}-1\}$ such that for all $r\in\{0,\ldots,N-2\}$ it holds 
$$G_2((m(N-1)+r)T_2)=H^2.$$ Then, the results of step 1, in particular estimate \eqref{e-suppISC}, apply with parameters $(T_2,\mu_2)$ and corresponding $C_2$;
\item or we replace $(T_2,\mu_2)$ with $(T_3,\mu_3)$, where $T_3:=(N-1)T_2$ and $\mu_3:=\frac{\mu_2}{N-1}$. We denote the associated connectivity graph with $G_3$ and observe that for all $m\in \{0,\ldots, 2^{N(N-1)}-1\}$ it holds 
$$G_3((m(N-1)+r)T_3)\in \{H^3,\ldots,H^{2^{N(N-1)}}\}.$$ We iterate again.
\end{itemize}
\end{enumerate}
Clearly, Case 2 occurs for no more than $2^{N(N-1)}-1$ times, as for each step one of the $H^k$ is removed. Then, Case 1 certainly occurs for some $(T_k,\mu_k)$ and corresponding $C_k,\delta_k$, with $k\in\{1,\ldots,2^{N(N-1)}\}$. It then holds
\begin{eqnarray*}
\diam{(N-1)2^NT}\leq  C_k\diam{0}
\end{eqnarray*}
for some $k$.  By observing that $\mu_kT_k=\mu T$, we have\\
\begin{eqnarray*}
C_k&=&1-\tfrac12\left( \tfrac{\underline{m}\mu T}{N+\underline{m}\mu T}    \right)^{N-1}\exp(-2\tfrac{(N-1)^2}{N}T_k\,\underline{m})\leq C
\end{eqnarray*}
with
\begin{eqnarray*}
C&=&1-\tfrac12\left( \tfrac{\underline{m}\mu T}{N+\underline{m}\mu T}    \right)^{N-1}\exp(-2\tfrac{(N-1)^{(2^{N(N-1)}+1)}}{N}T\,\underline{m}).
\end{eqnarray*}
By induction, this proves the result. \end{proof}

\subsection{Flocking}
In this section, we prove Theorem \ref{t-main2}. Also in this case, we prove it by providing an explicit rate of convergence for the velocity variables.

We first recall some properties of \eqref{e-ODE-2nd}, that are slight generalizations of the properties of the original cooperative second-order systems. With this goal, we recall the definition the diameters in the position and velocity variables, respectively:
\begin{eqnarray}
&&\D_X(t):=\max_{ij}|x_i(t)-x_j(t)|,\label{e-DX}\\
&&\D_V(t):=\max_{ij}|v_i(t)-v_j(t)|.\label{e-DV}
\end{eqnarray}
We remark the following fundamental property: the dynamics of the velocity variables $v_i$ in \eqref{e-ODE-2nd} is in reality of the form \eqref{e-ODE}, with the exception of the fact that the $\phi_{ij}$ do not depend on the $v_i$ variables, but on the $x_i$. Nevertheless, we have that results of Sections \ref{s-contractive} and \ref{s-reduction} apply, with no relevant differences in the proofs.

\begin{proposition} \label{p-2nd} Assume that hypotheses (H1)-(H2) in Corollary \ref{c-main} hold. Then, solutions of \eqref{e-ODE-2nd} satisfy the following properties:
\begin{enumerate}
\item solutions in the Carathéodory sense exist, globally in time, and are unique;
\item the support in the $v_i$ variables is contractive, i.e. 
$$\mathrm{supp}_V((x,v)(t)):=\mathrm{conv}(\{v_i(t)\})$$
satisfies $\mathrm{supp}_V((x,v)(t))\supseteq \mathrm{supp}_V((x,v)(s))$ for $0\leq t\leq s$;
\item the diameter in the velocity variable $\D_V$ is a nonincreasing function;
\item for all $t_0,t\geq 0$ it holds
\begin{equation}\label{e-DX-dyn}
\D_X(t_0+t)\leq \D_X(t_0)+t\D_V(t_0).
\end{equation}
\end{enumerate}
\end{proposition}
\begin{proof} Statements 1-2-3 are proved as in Section \ref{s-contractive}. We now prove Statement 4: let $t_0,t\geq 0$ be given and $i,j$ be a pair of indexes that satisfy $|x_i(t_0+t)-x_j(t_0+t)|=\D_X(t_0+t)$. Assume that $s\in[t_0,t_0+t]$ is a point of differentiability of~$x_i, x_j$, where $\D_X(s)=|x_i(s)-x_j(s)|$, and compute
\begin{eqnarray*}
\frac{d}{ds} (x_i(s)- x_j(s))^2&=&2 (x_i(s)-x_j(s))\cdot (v_i(s)-v_j(s))\leq 2\D_X(s)\D_V(s)\leq 2\D_X(s)\D_V(t_0),
\end{eqnarray*}
by using Statement 3. One now gets $\frac{d}{ds} D_X(s)\leq \D_V(t_0)$, from which the result follows.
\end{proof}

We are now ready to prove Theorem \ref{t-main2}. 

\begin{proof} We first aim to apply Corollary \ref{c-main} to the system of the $v_i$ variables. With this goal, we first recall that the interaction function $\lambda_i(x)\phi_{ij}(x_i,x_j)=1\cdot \phi(|x_i-x_j|)$ has a maximum $\overline{m}=\phi(0)$. We then define $\tau=d^* T/\overline{m}=d^* T/\phi(0)$. By \eqref{e-DX-dyn}, for all $t\in[0,(n+1)\tau]$ it holds
$$\D_X(t)\leq \D_X(0)+(n+1)\tau \D_V(0).$$
On such time interval $[0,\tau]$, the interaction function $\lambda_i(x)\phi_{ij}(x_i,x_j)$ is then bounded from below by
$$\underline{m}(n\tau):=1\cdot\phi(\D_X(0)+(n+1)\tau \D_V(0)).$$
Then, estimate \eqref{e-stimaV} is just the rewriting of \eqref{e-stimadiamThm} on the time interval $[n\tau,(n+1)\tau]$. 

We now prove the sufficient condition for alignment. We first rewrite \eqref{e-stimaV} as follows, by simply using $\log(1+x)\leq x$:
\begin{eqnarray}
&&\log\left(\frac{\D_V((n+1)\tau)}{\D_V(n\tau)}\right)\leq \log(C(n\tau))\leq - \left(\tfrac{\phi(\D_X(0)+(n+1)\tau \D_V(0))\mu T}{N+\phi(\D_X(0)+(n+1)\tau \D_V(0))\mu T}    \right)^{d^*}\exp(-2\tfrac{N-1}{N}d^*T\phi(0)),\nonumber\\
&&\leq - \phi(\D_X(0)+(n+1)\tau \D_V(0))^{d^*} K,\label{e-V2}
\end{eqnarray}
with $$K:=\left(\tfrac{\mu T}{N+\phi(0)\mu T}\right)^{d^*}\exp(-2\tfrac{N-1}{N}d^*T\phi(0)).$$

We now iteratively apply \eqref{e-V2} and find 
\begin{eqnarray*}
\log\left(\frac{\D_V(n\tau)}{\D_V(0)}\right)&\leq& -\sum_{j=1}^n \phi(\D_X(0)+j\tau \D_V(0))^{d^*} K\leq -K\int_{\tau\D_V(0)}^{(n+1)\tau\D_V(0)}\phi(\D_X(0)+r)^{d^*}\,dr,
\end{eqnarray*}where we estimated the series from above by the integral, due to the fact that $\phi$ is nonincreasing. If $\phi^{d^*}$ is nonintegrable at $+\infty$, then $\lim_{n\to+\infty}\log\left(\frac{\D_V(n\tau)}{\D_V(0)}\right)=-\infty$, hence $\lim_{n\to+\infty}\D_V(n\tau)=0$. By contractivity of the support in the velocity variable, this implies $\lim_{t\to+\infty}\D_V(t)=0$, i.e. alignment occurs.

We now aim to prove the sufficient condition for flocking. Observe that, by iteratively applying Proposition \ref{p-2nd}, Statement 4, one has
\begin{eqnarray*}
&&\D_X((n+1)\tau)\leq \D_X(0)+\tau \sum_{j=0}^n \D_V(j\tau)\leq   \D_X(0)+\tau \D_V(0)\left(1+ \sum_{j=1}^n \prod_{k=1}^j C(k\tau)\right).
\end{eqnarray*}
We then find a sufficient condition for the convergence of the series, that implies uniform boundedness of $\D_X(t)$, i.e. flocking. By Raabe's condition, one has that the series converges if
\begin{eqnarray*}
\lim_{n\to+\infty} n\left(\frac{\prod_{k=1}^n C(k\tau)}{\prod_{k=1}^{n+1} C(k\tau)}-1\right)> 1.
\end{eqnarray*}
A direct computation shows that this corresponds to
\begin{eqnarray}
\label{Raabe-cond}
\lim_{n\to+\infty} n \left(\frac{1}{C((n+1)\tau )}-1\right)>1. \label{e-raabe}
\end{eqnarray}
We now recall that $\phi$ is nonincreasing and positive, thus $L:=\lim_{r\to+\infty} \phi(r)$ exists and it satisfies $L\geq 0$. We have two cases:
\begin{itemize}
\item if $L>0$, then $\lim_{n\to+\infty} C(n\tau )<1$ and condition \eqref{e-raabe} is satisfied
\item if $L=0$, then Raabe's condition becomes
\begin{equation}
\label{e-raabe-L0}
\lim_{n\to+\infty} n  \left( \tfrac{\phi(n\tau \D_V(0))\mu T}{N}    \right)^{d^*}\exp(-2\tfrac{N-1}{N}d^*T\phi(0))>1,
\end{equation}
that is the required condition.
\end{itemize}
\end{proof}

\section{Examples and comparison with the literature}

\label{s-examples}

In this section, we provide some examples that show the interest of our result, and in particular sharpness of the rate of convergence. We also compare our results with other convergence conditions available in the literature.

We first consider first-order systems: we observe that our convergence result is sharp when the dynamics is linear, i.e. for $\underline{m}=\overline{m}$. For simplicity, we choose $\underline{m}=\overline{m}=1$ i.e. we focus on systems of the form \eqref{e-ODEmain}, that are
\begin{equation}
\dot x_i(t)=\frac{1}{N}\sum_{j=1}^N M_{ij}(t) (x_j-x_i)\mbox{~~with~~}M_{ij}:[0,+\infty)\mapsto [0,1].
 \end{equation}
 
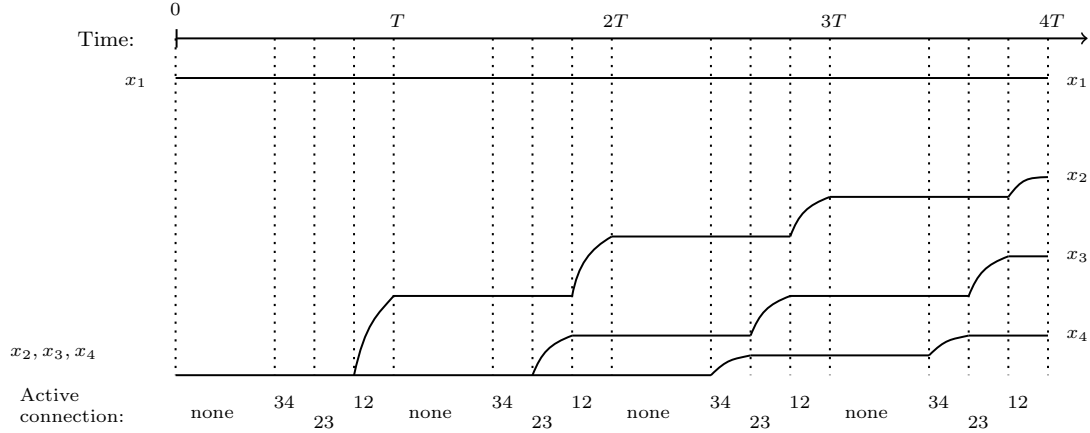
\begin{figure}[htb]

\tikzset{every picture/.style={line width=0.75pt},every node/.style={font=\scriptsize}} 

\begin{tikzpicture}[x=0.75pt,y=0.75pt,yscale=-1,xscale=1]

\draw[|-> ]    (90,10) -- (550,10) ;
\draw (55,10) node {\footnotesize Time:};
\draw (90,-5) node    {$0$};
\draw (202,1) node    {$T$};
\draw (312,1) node    {$2T$};
\draw (422,1) node    {$3T$};
\draw (532,1) node    {$4T$};

\draw    (90,180) -- (360,180) ;
\draw  [dash pattern={on 0.84pt off 2.51pt}]  (90,10) -- (90,180) ;
\draw  [dash pattern={on 0.84pt off 2.51pt}]  (140,10) -- (140,180) ;
\draw  [dash pattern={on 0.84pt off 2.51pt}]  (160,10) -- (160,180) ;
\draw  [dash pattern={on 0.84pt off 2.51pt}]  (180,10) -- (180,180) ;
\draw  [dash pattern={on 0.84pt off 2.51pt}]  (200,10) -- (200,180) ;
\draw  [dash pattern={on 0.84pt off 2.51pt}]  (250,10) -- (250,180) ;
\draw  [dash pattern={on 0.84pt off 2.51pt}]  (270,10) -- (270,180) ;
\draw  [dash pattern={on 0.84pt off 2.51pt}]  (290,10) -- (290,180) ;
\draw  [dash pattern={on 0.84pt off 2.51pt}]  (310,10) -- (310,180) ;
\draw  [dash pattern={on 0.84pt off 2.51pt}]  (360,10) -- (360,180) ;
\draw  [dash pattern={on 0.84pt off 2.51pt}]  (380,10) -- (380,180) ;
\draw  [dash pattern={on 0.84pt off 2.51pt}]  (400,10) -- (400,180) ;
\draw  [dash pattern={on 0.84pt off 2.51pt}]  (420,10) -- (420,180) ;
\draw  [dash pattern={on 0.84pt off 2.51pt}]  (470,10) -- (470,180) ;
\draw  [dash pattern={on 0.84pt off 2.51pt}]  (490,10) -- (490,180) ;
\draw  [dash pattern={on 0.84pt off 2.51pt}]  (510,10) -- (510,180) ;
\draw  [dash pattern={on 0.84pt off 2.51pt}]  (530,10) -- (530,180) ;
\draw    (90,30) -- (530,30) ;
\draw    (180,180) .. controls (185.67,153.5) and (192,149.17) .. (200,140) ;
\draw    (200,140) -- (290,140) ;
\draw    (290,140) .. controls (293,124) and (299,116.33) .. (310,110) ;
\draw    (310,110) -- (400,110) ;
\draw    (400,110) .. controls (404.33,98.33) and (408.67,94.33) .. (420,90) ;
\draw    (510,90) .. controls (517,80.67) and (519.33,80.33) .. (530,80) ;
\draw    (420,90) -- (510,90) ;
\draw    (270,180) .. controls (274.67,168.33) and (278.67,164.33) .. (290,160) ;
\draw    (290,160) -- (380,160) ;
\draw    (380,160) .. controls (384.67,148.33) and (388.67,144.33) .. (400,140) ;
\draw    (400,140) -- (490,140) ;
\draw    (490,140) .. controls (494.67,128.33) and (498.67,124.33) .. (510,120) ;
\draw    (380,170) .. controls (371,172) and (368.33,171.67) .. (360,180) ;
\draw    (380,170) -- (470,170) ;
\draw    (490,160) .. controls (481,162) and (478.33,161.67) .. (470,170) ;
\draw    (510,120) -- (530,120) ;
\draw    (490,160) -- (530,160) ;

\draw (545,32) node    {$x_{1}$};
\draw (10,185) node [anchor=north west][inner sep=0.75pt]   [align=left] { Active\\ connection:};
\draw (141-3,185+4) node [anchor=north west][inner sep=0.75pt]    {$34$};
\draw (161-3,195+4) node [anchor=north west][inner sep=0.75pt]    {$23$};
\draw (545,80) node {$x_{2}$};
\draw (181-3,185+4) node [anchor=north west][inner sep=0.75pt]    {$12$};
\draw (101-5,192+4) node [anchor=north west][inner sep=0.75pt]   [align=left] {none};
\draw (251-3,185+4) node [anchor=north west][inner sep=0.75pt]    {$34$};
\draw (271-3,195+4) node [anchor=north west][inner sep=0.75pt]    {$23$};
\draw (291-3,185+4) node [anchor=north west][inner sep=0.75pt]    {$12$};
\draw (211-5,192+4) node [anchor=north west][inner sep=0.75pt]   [align=left] {none};
\draw (361-3,185+4) node [anchor=north west][inner sep=0.75pt]    {$34$};
\draw (381-3,195+4) node [anchor=north west][inner sep=0.75pt]    {$23$};
\draw (401-3,185+4) node [anchor=north west][inner sep=0.75pt]    {$12$};
\draw (321-5,192+4) node [anchor=north west][inner sep=0.75pt]   [align=left] {none};
\draw (471-3,185+4) node [anchor=north west][inner sep=0.75pt]    {$34$};
\draw (491-3,195+4) node [anchor=north west][inner sep=0.75pt]    {$23$};
\draw (511-3,185+4) node [anchor=north west][inner sep=0.75pt]    {$12$};
\draw (431-5,192+4) node [anchor=north west][inner sep=0.75pt]   [align=left] {none};
\draw (545,120) node {$x_{3}$};
\draw (545,160) node {$x_{4}$};
\draw (70,32) node    {$x_{1}$};
\draw (5,165) node [anchor=north west][inner sep=0.75pt]   [align=right] {$x_{2} ,x_{3} ,x_{4}$};

\end{tikzpicture}
\caption{Example \ref{ex-1} with $N=4$.}
\end{figure}

\begin{example}[First-order system]\label{ex-1}We now fix a system of $N$ agents and two parameters $T,\mu>0$; for simplicity of description, we assume $\mu\leq \frac1{N-1}$ from now on, so that no more than one connection is active at each time. We aim to describe the ``worst case'' scenario for our estimate in Theorem \ref{t-main}.  It occurs when the connectivity graph is reduced to a single path and when interactions $M_{ij}(t)$ induce a very slow dynamics of agents. For this reason, we define the following interactions:
\begin{itemize}
\item it holds $M_{ij}\equiv 0$ for all $i,j\in\{1,\ldots,N\}$ except for $j=i-1$;
\item the $M_{i,i-1}$ functions are $T$-periodic, defined on $[0,T]$ by
\begin{equation}M_{i,i-1}(t):=\begin{cases}
1&\mbox{~for $t\in [T-i\mu-1,T-(i-1)\mu]$.}\\
0&\mbox{~elsewhere.}
\end{cases}\label{e-Mex}
\end{equation}
\end{itemize}
The resulting interaction graph is then the following:
$$1\leftarrow 2\leftarrow 3\leftarrow \ldots\leftarrow N.$$

We consider the following initial configuration: $x_1(0)=0$, and $x_j(0)=1$ for all $j\in\{2,\ldots,N\}$. The dynamics can be explicitly computed. We define $T_k:=T-k\mu$ and compute:
\begin{itemize}
\item $x_1(t)=0$ for all $t\geq 0$
\item for all $k\in\mathbb{N}$ it holds $x_2(kT+\tau)=$
$$\begin{cases}
x_2(kT)=\exp(-k\mu)&\mbox{ for $\tau\in[0,T_{1}]$,}\\
\exp(-(\tau-T_{1}))x_2(kT)&\mbox{ for $\tau\in[T_{1},T]$.}
\end{cases}$$

\item $x_3(\tau)=1$ for $\tau\in[0,T]$, while for $k\in\mathbb{N}\setminus\{0\}$ it holds $x_3(kT)=k\exp(-(k-1)\mu)-(k-1)\exp(-k\mu)$ as well as
$x_3(kT+\tau)=$
$$\hspace{-5mm}\begin{cases}
x_3(kT)&\mbox{ for $\tau\in[0,T_{2}]$,}\\
x_2(kT)+e^{-(\tau-T_{2})}(x_3(kT)-x_2(kT))&\mbox{for $\tau\in[T_{2},T_{1}]$,}\\
x_3((k+1)T)&\mbox{ for $\tau\in[T_{1},T]$.}
\end{cases}$$

\item for $j\in\mathbb{N}\setminus\{0,1,2\}$ we recursively have $x_j(\tau)=1$ for $\tau\in[0,(j-2)T]$ and for $k\in\mathbb{N}\setminus[0,(j-2))$ the following recursive formula holds:
\begin{equation}
x_j((k+1)T)=x_{j-1}(kT)+\exp(-\mu)(x_j(kT)-x_{j-1}(kT)).
\label{e-xjkT}
\end{equation}
Moreover, it always holds
$x_j(kT+\tau)=$
$$\hspace{-7mm}\begin{cases}
x_j(kT)&\mbox{ for $\tau\in[0,T_{j-1}]$,}\\
\begin{array}{l}x_{j-1}(kT)+\\
e^{-(\tau-T_{j-1})}(x_j(kT)-x_{j-1}(kT))\end{array}&\mbox{for $\tau\in[T_{j-1},T_{j-2}]$,}\\
x_j((k+1)T)&\mbox{ for $\tau\in[T_{j-2},T]$.}
\end{cases}$$
\end{itemize}

For a given $N$, this example shows the following key phenomenon: the graph length is $d^*=N-1$ and the diameter $\D(t)=\max_{i,j}|x_i(t)-x_j(t)|$ coincides with $|x_N(t)-x_1(t)|=x_N(t)$. In particular, the explicit computation of the trajectory shows that it holds $x_N((d^*-1)T)=x_N((N-2)T)=1=x_N(0)$, i.e. that there is no reduction of the diameter in the first $(d^*-1)$ periods of time.

While we found no closed form for the $x_j(kT)$, we can numerically investigate the discrepancy between the evolution of the diameter in this example and the estimate given by Theorem \ref{t-main}. In particular, in the case of $N$ particles, one has $\D(0)=1$ and $\D(d^*T)=x_N((N-1)T)$. We then compare the value of $\D(d^*T)$ with the value of $C\D(0)=C$ given by \eqref{e-C}. We choose $T=N$, $\mu=\frac{1}{N-1}$. Since all values are very close to 1, in Table \ref{t-dim1} we compare  $1-\D(d^*T)$ with $1-C$. It is remarkable to observe that the estimate given by Theorem \ref{t-main} becomes less and less accurate when $N$ grows. This is due to the fact that the estimate is based on the idea that all agents might interact at the end of the intervals, while in reality interaction occurs earlier and produces a stronger reduction of the support. This phenomenon has a larger impact for large $N$.

\begin{table}[htb]
\begin{centering}
\begin{tabular}{|c||c|c|c|c|c|c|}\hline
N & 3&4&5&6\\\hline
$1-\D(d^*T)$ &0.6035  & 0.3993&      0.2591  &   0.1666   \\\hline
$1-C\D(0)$ & 3.4D-7&   2.9D-13 & 3.3D-21&  5.6D-31    \\\hline
\end{tabular}

\vspace{2mm}

\begin{tabular}{|c||c|c|c|c|c|c|c|c|}\hline
N &7&8&9&10\\\hline
$1-\D(d^*T)$ &   0.1066&   0.0679 &   0.0432 &  0.0275\\\hline
$1-C\D(0)$ &1.4D-42 &  5.4D-56 & 3.3D-71&  3.3D-88\\\hline
\end{tabular}

\end{centering}

\vspace{2mm}

\caption{Example \ref{ex-1}: real trajectory vs estimate of Theorem \ref{t-main}.}
\label{t-dim1}
\end{table}

\end{example}

\begin{example}[Second-order system on the line]\label{ex-2} We now consider a system of $N$ particles on the real line, with dynamics given by the second-order system \ref{e-ODE-2nd}, i.e. the Cucker-Smale model with lack of connections, with given parameters $T,\mu>0$. We choose the original interaction function $\phi(r)=\frac{1}{(1+r^2)^{\beta}}$. Here, we do not aim to evaluate whether estimates about the diameters (in the space and/or velocity variables) are sharp. We are instead interested in showing that flocking is ensured only when the parameter $\beta$ in the definition of $\phi$ satisfy $\beta<1/(2d^*)$, i.e. when the function $r\cdot \phi(r)$ is not integrable at $+\infty$.

We choose the following interaction functions:
\begin{itemize}
\item it holds $M_{ij}\equiv 0$ for all $i,j\in\{1,\ldots,N\}$ with $j<i+1$;
\item it holds $M_{ij}\equiv 1$ for all $i,j\in\{2,\ldots,N\}$ with $j>i$; remark that the index $1$ is not included;
\item the $M_{i,i+1}$ functions are $T$-periodic, defined on $[0,T]$ by
\begin{equation}M_{i,i+1}(t):=\begin{cases}
1&\mbox{~for $t\in [T-i\mu-1,T-(i-1)\mu]$.}\\
0&\mbox{~elsewhere.}
\end{cases}\label{e-Mex2}
\end{equation}
\end{itemize}
 The corresponding interaction graph is given by Figure \ref{f-CSgraph}. It has length $d^*=N-1$: indeed, the only globally reachable node is 1 and the longest path to it is given by $N\to\ldots\to 1$, that has length $N-1$.

 \begin{figure}
\centering
\tikzset{every picture/.style={line width=0.75pt}} 

\begin{tikzpicture}[x=0.75pt,y=0.75pt,yscale=-1,xscale=1]

\draw    (110,50) -- (82,50) ;
\draw [shift={(80,50)}, rotate = 360] [color={rgb, 255:red, 0; green, 0; blue, 0 }  ][line width=0.75]    (10.93,-3.29) .. controls (6.95,-1.4) and (3.31,-0.3) .. (0,0) .. controls (3.31,0.3) and (6.95,1.4) .. (10.93,3.29)   ;
\draw    (148,50) -- (122,50) ;
\draw [shift={(120,50)}, rotate = 360] [color={rgb, 255:red, 0; green, 0; blue, 0 }  ][line width=0.75]    (10.93,-3.29) .. controls (6.95,-1.4) and (3.31,-0.3) .. (0,0) .. controls (3.31,0.3) and (6.95,1.4) .. (10.93,3.29)   ;
\draw [shift={(150,50)}, rotate = 180] [color={rgb, 255:red, 0; green, 0; blue, 0 }  ][line width=0.75]    (10.93,-3.29) .. controls (6.95,-1.4) and (3.31,-0.3) .. (0,0) .. controls (3.31,0.3) and (6.95,1.4) .. (10.93,3.29)   ;
\draw    (188,50) -- (162,50) ;
\draw [shift={(160,50)}, rotate = 360] [color={rgb, 255:red, 0; green, 0; blue, 0 }  ][line width=0.75]    (10.93,-3.29) .. controls (6.95,-1.4) and (3.31,-0.3) .. (0,0) .. controls (3.31,0.3) and (6.95,1.4) .. (10.93,3.29)   ;
\draw [shift={(190,50)}, rotate = 180] [color={rgb, 255:red, 0; green, 0; blue, 0 }  ][line width=0.75]    (10.93,-3.29) .. controls (6.95,-1.4) and (3.31,-0.3) .. (0,0) .. controls (3.31,0.3) and (6.95,1.4) .. (10.93,3.29)   ;
\draw    (228,50) -- (202,50) ;
\draw [shift={(200,50)}, rotate = 360] [color={rgb, 255:red, 0; green, 0; blue, 0 }  ][line width=0.75]    (10.93,-3.29) .. controls (6.95,-1.4) and (3.31,-0.3) .. (0,0) .. controls (3.31,0.3) and (6.95,1.4) .. (10.93,3.29)   ;
\draw [shift={(230,50)}, rotate = 180] [color={rgb, 255:red, 0; green, 0; blue, 0 }  ][line width=0.75]    (10.93,-3.29) .. controls (6.95,-1.4) and (3.31,-0.3) .. (0,0) .. controls (3.31,0.3) and (6.95,1.4) .. (10.93,3.29)   ;
\draw    (280,50) -- (252,50) ;
\draw [shift={(250,50)}, rotate = 360] [color={rgb, 255:red, 0; green, 0; blue, 0 }  ][line width=0.75]    (10.93,-3.29) .. controls (6.95,-1.4) and (3.31,-0.3) .. (0,0) .. controls (3.31,0.3) and (6.95,1.4) .. (10.93,3.29)   ;
\draw    (120,60) .. controls (155.46,71.82) and (150.17,71.03) .. (188.23,60.49) ;
\draw [shift={(190,60)}, rotate = 164.62] [color={rgb, 255:red, 0; green, 0; blue, 0 }  ][line width=0.75]    (10.93,-3.29) .. controls (6.95,-1.4) and (3.31,-0.3) .. (0,0) .. controls (3.31,0.3) and (6.95,1.4) .. (10.93,3.29)   ;
\draw    (120,60) .. controls (154.48,79.7) and (200.59,78.06) .. (238.28,60.8) ;
\draw [shift={(240,60)}, rotate = 154.65] [color={rgb, 255:red, 0; green, 0; blue, 0 }  ][line width=0.75]    (10.93,-3.29) .. controls (6.95,-1.4) and (3.31,-0.3) .. (0,0) .. controls (3.31,0.3) and (6.95,1.4) .. (10.93,3.29)   ;
\draw    (120,60) .. controls (157.43,83.64) and (239.49,78.17) .. (278.26,60.8) ;
\draw [shift={(280,60)}, rotate = 154.65] [color={rgb, 255:red, 0; green, 0; blue, 0 }  ][line width=0.75]    (10.93,-3.29) .. controls (6.95,-1.4) and (3.31,-0.3) .. (0,0) .. controls (3.31,0.3) and (6.95,1.4) .. (10.93,3.29)   ;
\draw    (160,40) .. controls (195.46,25.07) and (190.17,26.07) .. (228.23,39.38) ;
\draw [shift={(230,40)}, rotate = 199.16] [color={rgb, 255:red, 0; green, 0; blue, 0 }  ][line width=0.75]    (10.93,-3.29) .. controls (6.95,-1.4) and (3.31,-0.3) .. (0,0) .. controls (3.31,0.3) and (6.95,1.4) .. (10.93,3.29)   ;
\draw    (160,40) .. controls (194.48,15.11) and (240.59,17.19) .. (278.28,38.99) ;
\draw [shift={(280,40)}, rotate = 210.9] [color={rgb, 255:red, 0; green, 0; blue, 0 }  ][line width=0.75]    (10.93,-3.29) .. controls (6.95,-1.4) and (3.31,-0.3) .. (0,0) .. controls (3.31,0.3) and (6.95,1.4) .. (10.93,3.29)   ;

\draw (71,42.4) node [anchor=north west][inner sep=0.75pt]    {$1$};
\draw (111,42.4) node [anchor=north west][inner sep=0.75pt]    {$2$};
\draw (151,42.4) node [anchor=north west][inner sep=0.75pt]    {$3$};
\draw (191,42.4) node [anchor=north west][inner sep=0.75pt]    {$4$};
\draw (234,42.4) node [anchor=north west][inner sep=0.75pt]    {$...$};
\draw (281,42.4) node [anchor=north west][inner sep=0.75pt]    {$N$};

\end{tikzpicture}
\label{f-CSgraph}
\caption{Example \ref{ex-2}: interaction graph.}
\end{figure}
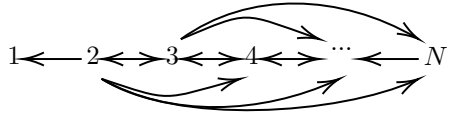

We choose the initial condition to be $x_1(0)=v_1(0)=0$ and $x_j(0)=v_j(0)=1$ for all $j\in\{2,\ldots,N\}$. One can easily prove that the order given by indexes is preserved by the dynamics in both the position and velocity variables, i.e. that $x_i(t)\leq x_j(t)$ and $v_i(t)\leq v_j(t)$ for all $t\geq 0$ and $i<j$. Moreover, the choice of $M_{ij}$ given above has the following consequences:
\begin{itemize}
\item it holds $x_1(t)=v_1(t)=0$ for all $t\geq 0$, since it holds $M_{1j}\equiv 0$ for all $j\in\{2,\ldots,N\}$;
\item for all indexes $j\in \{2,\ldots,N\}$, there is a weak interaction towards the index $j-1$ and strong constant interactions towards indexes $l<j$.
\end{itemize}
While we found no explicit formula for trajectories, we numerically compute them. We fix $\beta=1/10$ and study the cases of$N\in \{4,6,8\}$. See Figure \ref{f-ex2traj}. It is clear that one has flocking for $N=4$, while for $N=8$ there is no convergence in either variables. Both examples agree with the sufficient condition $\beta<\frac{1}{2d^*}$ for flocking, since $d^*=N-1$ in this case.  The most interesting case is given by $N=6$ (see Figure \ref{f-ex2traj} center), for which it holds $\beta=\frac{1}{2d^*}$: here, convergence of the velocity variables (bottom figure) is ensured, but boundedness of the position variables (top figure) is not verified.

\begin{figure}
\includegraphics[width=\linewidth]{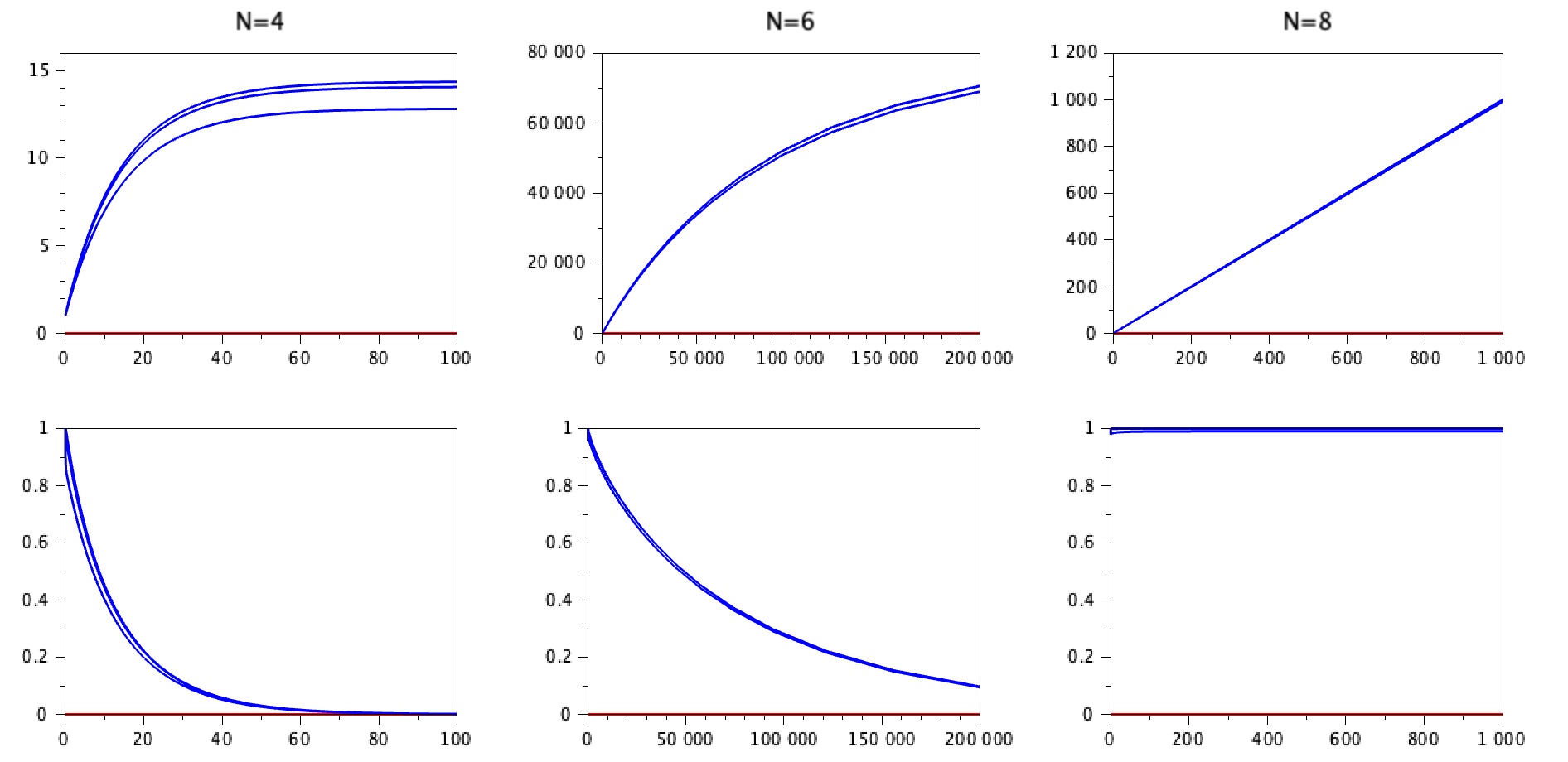}

\label{f-ex2traj}
\caption{Example \ref{ex-2}: trajectories of the system for $N\in \{4,6,8\}$ with position (top) and velocity (bottom)}
\end{figure}

\end{example}

\subsection{Comparison with the literature}
\label{s-comparison}

In this section, we provide some comparisons with available results in the literature. We basically have two quite different points of views: the symmetric and the non-symmetric ones.

In several key results, one imposes some form of symmetry of the communication weights. A strong form of symmetry, of the kind $M_{ij}=M_{ji}$, is explicitly imposed in several contributions, e.g. in \cite{chowdhury2018estimation} for first-order linear systems, in \cite{chowdhury2015consensus} for linear integrators, and in  \cite{bonnet2021consensus} for Cucker-Smale systems. A much weaker (thus more general) symmetry condition is known as the cut-balance condition, see \cite{cut-balance}: there exists $K>0$ such that for all subsets of agents $S\subset \{1,\ldots, N \}$ and for all $t > 0$ one requires
$$ \sum_{i\in S,j\not\in S} M_{ij}(t) \leq K \sum_{i\in S,j\not\in S} M_{ji}(t).$$ In \cite{arc-balance}, a generalization, known as the arc-balance condition, is introduced. In \cite{cut-balance2}, the result is extended to allow for non-instantaneous reciprocity.
Even though these results are fundamental, this symmetric point of view is somehow opposite to ours.

Indeed, the non-symmetric point of view (dating back to Moreau in \cite{moreau2004stability}) takes advantage of the direction of interactions: as the most explicit example, in his original contribution, Moreau builds a directed graph and requires a globally reachable node. This point of view, less discussed in the literature, can be encoded in the Integral Scrambling Condition discussed in Section \ref{s-ISC}. A more general condition of convergence, similar to the one studied here, is presented in \cite{caravenna}. There, a directed graph is built based on the following condition for the existence of the edge $i\to j$:
\begin{eqnarray*}&&\lim_{n\to +\infty} M_{ij}(t_n+t)= M_{ij}^*(t)\qquad \mbox{~with~} \int_t^{+\infty} M_{ij}^*(s)\,ds>0 \mbox{~for all $t>0$.}
\end{eqnarray*}

If the directed graph has a globally reachable node, then the first-order system converges to consensus. Under such a condition, no rate of convergence to consensus is available; indeed, one can introduce large intervals with no connection and still have the condition verified (by increasing the $t_n$).

\section{Conclusions}
\label{s-conclusions}

In this article, we have proved sufficient conditions for convergence of first-order systems to consensus and of second-order systems to flocking, under communication failures. We introduced quantitative estimates about the minimum level of service, namely the PE and ISC conditions. We proved that, for each of these conditions, consensus and flocking can be achieved under the classical conditions for systems with no communication failures. Yet, the rate of convergence is slower. In both consensus and flocking problems, we provided explicit rates of convergence.

In the future, we aim to find even weaker conditions for communication failures, to understand the (theoretical) minimum level of service ensuring consensus or flocking. Moreover, we aim to find explicit rates of convergence for other conditions, such as the algebraic connectivity (see \cite{bonnet2023consensus}).

\bibliographystyle{ieeetr}
\bibliography{refPE}

\begin{thebibliography}{10}

\bibitem{bullo2009distributed}
F.~Bullo, J.~Cort{\'e}s, and S.~Martinez, {\em Distributed control of robotic
  networks}.
\newblock Princeton University Press, 2009.

\bibitem{nedich2015convergence}
A.~Nedich, {\em Convergence rate of distributed averaging dynamics and
  optimization in networks}.
\newblock Now Publishers, Inc., 2015.

\bibitem{mesbahi2010graph}
M.~Mesbahi and M.~Egerstedt, {\em Graph theoretic methods in multiagent
  networks}.
\newblock Princeton University Press, 2010.

\bibitem{HK}
R.~Hegselmann and U.~Krause, ``Opinion dynamics and bounded confidence: models,
  analysis and simulation,'' {\em Journal of Artificial Societies and Social
  Simulation}, vol.~5, no.~3, p.~2, 2002.

\bibitem{rateHK}
S.~Mohajer and B.~Touri, ``On convergence rate of scalar hegselmann-krause
  dynamics,'' in {\em 2013 American control conference}, pp.~206--210, IEEE,
  2013.

\bibitem{rateHK2}
A.~Bhattacharyya, M.~Braverman, B.~Chazelle, and H.~L. Nguyen, ``On the
  convergence of the hegselmann-krause system,'' in {\em Proceedings of the 4th
  conference on Innovations in Theoretical Computer Science}, pp.~61--66, 2013.

\bibitem{CS}
F.~Cucker and S.~Smale, ``Emergent behavior in flocks,'' {\em IEEE Trans.
  Autom. Control}, vol.~52, no.~5, pp.~852--862, 2007.

\bibitem{watts1998collective}
D.~J. Watts and S.~H. Strogatz, ``Collective dynamics of
  ‘small-world’networks,'' {\em Nature}, vol.~393, no.~6684, pp.~440--442,
  1998.

\bibitem{olfati2007consensus}
R.~Olfati-Saber, J.~A. Fax, and R.~M. Murray, ``Consensus and cooperation in
  networked multi-agent systems,'' {\em Proceedings of the IEEE}, vol.~95,
  no.~1, pp.~215--233, 2007.

\bibitem{moreau2004stability}
L.~Moreau, ``Stability of continuous-time distributed consensus algorithms,''
  in {\em 43rd IEEE Conference on Decision and Control}, vol.~4,
  pp.~3998--4003, IEEE, 2004.

\bibitem{narendra2012stable}
K.~S. Narendra and A.~M. Annaswamy, {\em Stable adaptive systems}.
\newblock Courier Corporation, 2012.

\bibitem{ChSi2010}
Y.~Chitour and M.~Sigalotti, ``On the stabilization of persistently excited
  linear systems,'' {\em SIAM Journal on Control and Optimization}, vol.~48,
  no.~6, pp.~4032--4055, 2010.

\bibitem{ChSi2014}
Y.~Chitour, F.~Colonius, and M.~Sigalotti, ``Growth rates for persistently
  excited linear systems,'' {\em Mathematics of Control, Signals, and Systems},
  vol.~26, no.~4, pp.~589--616, 2014.

\bibitem{ren2008distributed}
W.~Ren and R.~W. Beard, {\em Distributed consensus in multi-vehicle cooperative
  control}.
\newblock Springer, 2008.

\bibitem{tang2020bearing}
Z.~Tang, R.~Cunha, T.~Hamel, and C.~Silvestre, ``Bearing leader-follower
  formation control under persistence of excitation,'' {\em IFAC-PapersOnLine},
  vol.~53, no.~2, pp.~5671--5676, 2020.

\bibitem{manfredi2016criterion}
S.~Manfredi and D.~Angeli, ``A criterion for exponential consensus of
  time-varying non-monotone nonlinear networks,'' {\em IEEE Trans. Autom.
  Control}, vol.~62, no.~5, pp.~2483--2489, 2016.

\bibitem{bonnet2021consensus}
B.~Bonnet and E.~Flayac, ``Consensus and flocking under communication failures
  for a class of {C}ucker--{S}male systems,'' {\em Systems \& Control Letters},
  vol.~152, p.~104930, 2021.

\bibitem{anderson2016convergence}
B.~D. Anderson, G.~Shi, and J.~Trumpf, ``Convergence and state reconstruction
  of time-varying multi-agent systems from complete observability theory,''
  {\em IEEE TAC}, vol.~62, no.~5, pp.~2519--2523, 2016.

\bibitem{chaillet2008uniform}
A.~Chaillet, Y.~Chitour, A.~Lor{\'\i}a, and M.~Sigalotti, ``Uniform
  stabilization for linear systems with persistency of excitation: the
  neutrally stable and the double integrator cases,'' {\em Mathematics of
  Control, Signals, and Systems}, vol.~20, no.~2, pp.~135--156, 2008.

\bibitem{bonnet2023consensus}
B.~Bonnet-Weill and M.~Sigalotti, ``Exponential consensus formation in
  time-varying multiagent systems via compactification methods,'' in {\em 2024
  IEEE 63rd Conference on Decision and Control (CDC)}, pp.~5409--5416, 2024.

\bibitem{bonnet2022consensus}
B.~Bonnet, N.~Pouradier~Duteil, and M.~Sigalotti, ``Consensus formation in
  first-order graphon models with time-varying topologies,'' {\em Mathe. Mode.
  Meth. Appl. Sci.}, vol.~32, no.~11, pp.~2121--2188, 2022.

\bibitem{filippov}
A.~F. Filippov, ``Differential equations with discontinuous right-hand side,''
  {\em Matematicheskii sbornik}, vol.~93, no.~1, pp.~99--128, 1960.

\bibitem{smith}
H.~L. Smith, {\em Monotone dynamical systems}.
\newblock AMS, 1995.

\bibitem{rainer2002opinion}
U.~Krause and R.~Hegselmann, ``Opinion dynamics and bounded confidence: models,
  analysis and simulation,'' {\em Journal of Artificial Societies and Social
  Simulation}, vol.~5, no.~3, p.~2, 2002.

\bibitem{degroot1974reaching}
M.~H. DeGroot, ``Reaching a consensus,'' {\em Journal of the American
  Statistical association}, vol.~69, no.~345, pp.~118--121, 1974.

\bibitem{motsch2011new}
S.~Motsch and E.~Tadmor, ``A new model for self-organized dynamics and its
  flocking behavior,'' {\em J. Stat. Phys.}, vol.~144, pp.~923--947, 2011.

\bibitem{chowdhury2018estimation}
N.~R. Chowdhury, S.~Sukumar, M.~Maghenem, and A.~Lor{\'\i}a, ``On the
  estimation of the consensus rate of convergence in graphs with persistent
  interconnections,'' {\em International Journal of Control}, vol.~91, no.~1,
  pp.~132--144, 2018.

\bibitem{chowdhury2015consensus}
N.~R. Chowdhury and S.~Sukumar, ``Consensus analysis of double integrator
  agents with persistent interaction graphs,'' in {\em 2015 5th Australian
  Control Conference (AUCC)}, pp.~120--125, IEEE, 2015.

\bibitem{cut-balance}
J.~M. Hendrickx and J.~N. Tsitsiklis, ``Convergence of type-symmetric and
  cut-balanced consensus seeking systems,'' {\em IEEE Transactions on Automatic
  Control}, vol.~58, no.~1, pp.~214--218, 2013.

\bibitem{arc-balance}
G.~Shi and K.~H. Johansson, ``The role of persistent graphs in the agreement
  seeking of social networks,'' {\em IEEE Journal on Selected Areas in
  Communications}, vol.~31, no.~9, pp.~595--606, 2013.

\bibitem{cut-balance2}
S.~Martin and J.~M. Hendrickx, ``Continuous-time consensus under
  non-instantaneous reciprocity,'' {\em IEEE Transactions on Automatic
  Control}, vol.~61, no.~9, pp.~2484--2495, 2015.

\bibitem{caravenna}
M.~Bentaibi, L.~Caravenna, J.-P.~A. Gauthier, and F.~Rossi, ``Consensus in
  multiagent systems under communication failure,'' 2025.

\end{thebibliography}
\end{document}